\documentclass[reqno]{amsart}
\usepackage{mathrsfs}
\usepackage{graphicx}
\usepackage{amsmath}
\usepackage{amscd}
\usepackage{cite}
\usepackage{hyperref}
\usepackage{epsfig}
\usepackage{subfigure}
\usepackage{caption}
\usepackage{tikz}
  \usepackage{float}
  \usepackage{amssymb}
\usepackage{amsfonts}
\usepackage[all,cmtip]{xy}
\usepackage{appendix}
\usepackage{bm}
\usepackage{ulem}
\usepackage{tipa}
\usepackage[T1]{fontenc}
\usepackage[utf8]{inputenc}
\usepackage{geometry}
\newtheorem{Theorem}{Theorem}[section]
\newtheorem{Example}[Theorem]{Example}

\newtheorem{Proposition}[Theorem]{Proposition}
\newtheorem{Lemma}[Theorem]{Lemma}
\newtheorem{Corollary}[Theorem]{Corollary}

\newtheorem{Remark}[Theorem]{Remark}

\newcommand{\pd}{\partial}

\newcommand{\bC}{{\mathbb C}}

\newcommand{\bN}{{\mathbb N}}
\newcommand{\bP}{{\mathbb P}}
\newcommand{\bQ}{{\mathbb Q}}

\newcommand{\bZ}{{\mathbb Z}}

\newcommand{\cD}{{\mathcal D}}

\newcommand{\cP}{{\mathcal P}}

\newcommand{\half}{\frac{1}{2}}

\newcommand{\be}{\begin{equation}}
\newcommand{\ee}{\end{equation}}
\newcommand{\bea}{\begin{eqnarray}}
\newcommand{\ben}{\begin{eqnarray*}}
\newcommand{\een}{\end{eqnarray*}}
\newcommand{\eea}{\end{eqnarray}}

\DeclareMathOperator{\Aut}{Aut}

\DeclareMathOperator{\Res}{Res}

\DeclareMathOperator{\Hom}{Hom}

\usepackage{color}

\definecolor{yellow}{rgb}{1,1,0}
\definecolor{orange}{rgb}{1,.7,0}
\definecolor{red}{rgb}{1,0,0}
\definecolor{green}{rgb}{0,1,1}
\definecolor{white}{rgb}{1,1,1}

\definecolor{A}{rgb}{.75,1,.75}

\begin{document}

\newtheorem{myDef}{Definition}
\newtheorem{thm}{Theorem}
\newtheorem{eqn}{equation}

\title[$\mathbb N$-coefficient binomial polynomiality of Hurwitz numbers]
{On $\mathbb N$-Coefficient Binomial Polynomiality of\\Hurwitz Numbers and Generalized Dessin Counting}

\author{Zhiyuan Wang}
\address{Zhiyuan Wang, School of Mathematics and Statistics,
	Huazhong University of Science and Technology,
	Wuhan, China}
\email{wangzy23@hust.edu.cn}

\author{Chenglang Yang}
\address{Chenglang Yang, Institute for Math and AI, Wuhan University, Wuhan, China}
\address{Hua Loo-Keng Center for Mathematical Sciences,
	Academy of Mathematics and Systems Science,
	Chinese Academy of Sciences,
	Beijing, China}
\email{yangcl@pku.edu.cn}

\begin{abstract}
In this paper,
we study a certain type of Hurwitz numbers which count branched covers over the Riemann sphere admitting several branch points with fixed ramification types,
one branch point with a fixed number of preimages,
and one branch point with an arbitrary ramification type.
We prove that the dependence of this kind of Hurwitz numbers on parts of the ramification type over the last point is a polynomial.
Moreover,
when expanding this polynomial in terms of products of binomial coefficients,
we show that the coefficients are always non-negative integers via a pure combinatorial method.
Our result generalizes the polynomiality in several models,
including the one-part double Hurwitz numbers studied by Goulden-Jackson-Vakil,
the one-part double Hurwitz numbers with completed cycles studied by Shadrin-Spitz-Zvonkine,
and the generalized dessin counting.
\end{abstract} 

\maketitle


\section{Introduction}

The enumeration of branched coverings of Riemann surfaces
is one of the most important subjects in the study of moduli spaces of Riemann surfaces,
enumerative geometry and mathematical physics.
Fixing the number of branch points
and the ramification type over each branch point,
the weighted counting of isomorphism classes of such branched coverings is called the
Hurwitz number \cite{hur}.
These Hurwitz numbers are known to be closely related to many other theories
in mathematics and theoretical physics,
such as representation theory, integrable systems,
geometry of moduli spaces,
and topological string theory,
see e.g.
\cite{bcd,bm,bdks,cm,cd,cdo,di,dm,dyz,dmss,dkoss,elsv1,elsv2,gj,gjv,gjv2,gkl,gkls,jwy,kl,ok,op,op2,pa}
and references therein.

The polynomiality is one of the
most interesting features of the Hurwitz numbers.
In literatures,
the polynomial structures of various types of Hurwitz numbers are widely studied
using different methods
which are known to have deep connections to
the geometry of moduli spaces and Gromov-Witten theory.
The polynomiality of the single Hurwitz numbers is a consequence of the famous ELSV formula
(see \cite{elsv1,elsv2}, and see \cite{dkoss} for a new proof)
which relates the single Hurwitz numbers to
the intersection theory on the moduli spaces of curves.
The polynomiality of the double Hurwitz numbers was proved by Goulden-Jackson-Vakil \cite{gjv2},
which motivates some conjectural relations between
one-part double Hurwitz numbers and the intersection theory
on some suitable compactification of the Picard variety.
Another proof of the polynomiality of double Hurwitz numbers
was given by Johnson \cite{jo},
and his method is to use the infinite-wedge representation
of the double Hurwitz numbers proposed by Okounkov \cite{ok}
which reveals some connections between Hurwitz numbers and integrable hierarchies.
In \cite{ssz},
Shadrin-Spitz-Zvonkine proved the polynomiality
of the double Hurwitz numbers with completed cycles using the infinite-wedge representation
in this case.
The polynomiality of the monotone Hurwitz numbers and mixed double Hurwitz numbers
are proved by Goulden, Guay-Paquet, and Novak
in \cite{ggn1} and \cite{ggn2} respectively,
and see also \cite{hkl} for a further generalization.
Similar structures also appear in the case of spin Hurwitz numbers \cite{eop},
which are known to be a generalization of the ordinary Hurwitz numbers to type $B$.
In \cite{gkl},
Giacchetto-Kramer-Lewa\'nski proved the polynomiality of the spin double Hurwitz numbers
using the BKP integrability.

Now in this work,
we study the polynomiality of Hurwitz numbers of a certain type
using a combinatorial method.
We represent these Hurwitz numbers as a linear combination of
certain products of binomial coefficients with non-negative integral coefficients.
Such type of representations is inspired by the work \cite{st} of Stanley,
and see also \cite{hx, ol} for some generalizations of his work.
Fix a positive integer $r$ and partitions of integers $\lambda^1,\lambda^2,\cdots,\lambda^r$.
Let $\mu = (\mu_1,\cdots,\mu_l)$ be a partition with a fixed length $l$ and
satisfying $|\mu|\geq |\lambda^i|$ for every $i$,
and let $\tilde\lambda^i = (\lambda^i, 1^{|\mu|-|\lambda^i|})$ be the partition
obtained by adding some $1$ to $\lambda^i$ such that $|\tilde\lambda^i| = |\mu|$.
Fix a positive integer $k$,
and denote:
\be
H_{\lambda^1,\cdots,\lambda^r;k}^\circ (\mu)
= \sum_{\substack{ \sigma\in \cP_{|\mu|}, \text{ s.t. } l(\sigma) = k }} H_{g}^\circ
(\tilde\lambda^1,\cdots,\tilde\lambda^r,\sigma,\mu),
\ee
where $H_{g}^\circ (\tilde\lambda^1,\cdots,\tilde\lambda^r,\sigma,\mu)$ is the Hurwitz number
counting branched coverings from a connected Riemann surface of genus $g$ to the Riemann sphere
with $(r+2)$ given branch points,
such that the ramification types are given by the partitions $\tilde\lambda^1,\cdots,\tilde\lambda^r,\sigma,\mu$.
Here $\cP_{|\mu|}$ is the set of all partitions of the integer $|\mu|$,
and $l(\sigma)$ is the length of the partition $\sigma$.
Our main theorem is:
\begin{Theorem}
\label{thm-intro-main}
Fix three positive integers $r,k,l>0$ and $r$ partitions of integers $\lambda^1,\cdots,\lambda^r$.
Let $\mu = (\mu_1,\cdots,\mu_l)$ be a partition of length $l$ with $|\mu|\geq |\lambda^i|$ for every $i$,
then
\begin{equation*}
z_\mu \cdot H_{\lambda^1,\cdots,\lambda^r;k}^\circ (\mu)
\end{equation*}
is a linear combination of
\begin{equation*}
\binom{\mu_1}{d_1} \binom{\mu_2}{d_2} \cdots \binom{\mu_l}{d_l}
\end{equation*}
with non-negative integer coefficients,
where $d_j\geq 1$ for every $1\leq j\leq l$, and
\begin{equation*}
\sum_{j=1}^l d_j \leq \sum_{i=1}^r 2 \big( |\lambda^i| - l(\lambda^i)\big).
\end{equation*}
The number $z_\mu$ for a partition $\mu$ is defined by:
$z_\mu = \prod_{i\geq 1} m_i(\mu)! \cdot i^{m_i(\mu)}$
where $m_i(\mu)$ is the number of $i$ appearing in $\mu$.
\end{Theorem}

The non-negative integer coefficients of the above linear combination have a combinatorial interpretation
which arises naturally from the group-theoretic definition of  Hurwitz numbers,
see \S \ref{thm-binompol-H} for details.
As a straightforward consequence,
we see that $z_\mu \cdot H_{\lambda^1,\cdots,\lambda^r;k}^\circ (\mu)$
is a polynomial with rational coefficients in $\mu_1,\cdots,\mu_l$.
It is worth mentioning that the same result also holds
for the disconnected Hurwitz numbers $H_{\lambda^1,\cdots,\lambda^r;k}^\bullet (\mu)$
by the same argument,
see Remark \ref{rmk-disconnpol}.

As applications,
we apply the above theorem to two models
and obtain the polynomial structures of these two cases.
The first model is the one-part (and also $m$-part in general) double Hurwitz numbers with completed cycles.
The ordinary double Hurwitz number labeled by two partitions $\mu^\pm$
enumerates the branched coverings of the Riemann sphere which have ramification type $\mu^\pm$
over two given points
and simple ramifications over other branch points.
The double Hurwitz numbers with completed cycles
are defined similarly as the ordinary double Hurwitz numbers such that
the simple branches are all replaced by the completed cycles \cite{op}.
See \cite{gjv2} and \cite{ssz} for the polynomiality in these two cases.
As a corollary of the main theorem,
we have
(see \S \ref{sec-onepart} for details and a review of notations):
\begin{Proposition}
Fix a positive integer $r$.
Let $\mu = (\mu_1,\cdots,\mu_l)$ be a partition of fixed length $l$
with $|\mu| \geq r+1$,
and denote by $H_{g;\mu,\nu}^{(r)}$ the (connected or disconnected) double Hurwitz number
with completed $(r+1)$-cycles labeled by two partitions $\mu,\nu$.
Then the following $m$-part double Hurwitz numbers with completed cycles:
\begin{equation*}
H_{g;m;\mu}^{(r)} = \sum_{l(\nu) = m} H_{g;\mu,\nu}^{(r)}.
\end{equation*}
is an element in $\bQ [\mu_1,\cdots,\mu_l]$.
\end{Proposition}

The special case $m=1$, $r=1$ gives the polynomiality of the ordinary one-part double Hurwitz numbers firstly proved by Goulden-Jackson-Vakil \cite{gjv2},
and in this case,
we also obtain the property that when expanding the corresponding polynomial as a linear combination of products of binomial coefficients,
the coefficient will be nonnegative integers.
The special case $m=1$ of this proposition
gives the polynomiality of the one-part double Hurwitz numbers with completed cycles proved by Shadrin-Spitz-Zvonkine \cite{ssz}.

Another application of the main theorem is the polynomiality
of the generalized dessin counting.
Among various types of branched coverings of Riemann surfaces,
the branched coverings of the Riemann sphere over three points $0,1,\infty \in \bC\bP^1$
are of particular interest
since it leads to the notion of Grothendieck's dessin d'enfants \cite{gr, be},
which are bi-colored ribbon graphs on surfaces.
For an introduction to Grothendieck's dessin d'enfants, ribbon graphs,
and some related topics,
see the book \cite{lz}.
The enumeration of dessin d'enfants has been well-studied by mathematical physicists,
see e.g. \cite{ammn, ac, dmss, kz, kls, yz, zhou2, zhou3, zog}.
Let the monodromy type over $\infty\in \bC\bP^1$ and
the numbers of the preimages of $0$ and $1$ be fixed,
then the generating series of the weighted counting of such branched coverings over $\bC\bP^1$
is called the dessin partition function by Kazarian and Zograf in \cite{kz}.
In their paper,
they have derived the Virasoro constraints and Eynard-Orantin topological recursion
for the dessin partition function using the combinatorial properties of ribbon graphs.
As a consequence of the Virasoro constraints,
they obtain a cut-and-join representation for the dessin partition function
which naturally leads to the KP integrability (see also \cite{zog}).
Now in this work we study a generalization of the dessin partition function.
Fix a positive integer $r$,
and consider the branched coverings over $\bC\bP^1$ with $r+1$ branch points
$q_1,\cdots,q_r,q_{r+1} \in \bC\bP^1$,
such that the preimage $f^{-1}(q_i)$ consists of $k_i$ points for $1\leq i\leq r$,
and the ramification type over $q_{r+1}$ is described by
a partition of integer $\mu = (\mu_1,\cdots,\mu_l)$.
Denote by
\begin{equation*}
N_{k_1,\cdots,k_r}^\bullet (\mu) = \sum_{f} \frac{1}{|\Aut(f)|}
\end{equation*}
the weighted counting of equivalence classes of such branched coverings
from a (not necessarily connected) Riemann surface to $\bC\bP^1$.
The case $r=2$ recovers the dessin counting studied by Kazarian-Zograf \cite{kz}.
As a corollary of the main theorem,
we have (see \S \ref{sec-gdessin}):
\begin{Theorem}
Fix some positive integers $l>0$, $r>1$ and $k_1,k_2,\cdots,k_r>0$.
Then for a partition of integer $\mu = (\mu_1, \mu_2, \cdots, \mu_l)$ (with $|\mu|>k_i$ for every $i$),
the number
\begin{equation*}
z_\mu \cdot N_{|\mu|-k_1,\cdots,|\mu|-k_{r-1},k_r}^\circ (\mu)
\end{equation*}
is a linear combination of
\begin{equation*}
\binom{\mu_1}{d_1} \binom{\mu_2}{d_2} \cdots \binom{\mu_l}{d_l}
\end{equation*}
with non-negative integer coefficients,
where $d_j\geq 1$ and
$\sum_{j=1}^l d_j \leq 2\sum_{i=1}^{r-1} k_i$.
\end{Theorem}

In a previous version of this work
we proposed the polynomiality of the generalized dessin counting
as a conjecture,
and later it was pointed out by an anonymous referee that
this can be proved by applying the Eynard-Orantin topological recursion in this case \cite{bdks}.
After that, we found that the combinatorial method we use now can actually prove
the $\bN$-coefficient binomial polynomiality in a more general case
as in Theorem \ref{thm-intro-main}.

We also give some combinatorial properties of the generalized dessin partition function,
which is the following generating series of the generalized dessin counting:
\begin{equation*}
Z_{(r)}(s;v_1,\cdots,v_r;\bm p)
= \sum_{k_1,\cdots,k_r = 1}^\infty \sum_{\mu\in \cP}
\Big(\prod_{i=1}^r v_i^{k_i}\Big) s^{|\mu|} N_{k_1,\cdots,k_r}^\bullet (\mu) \cdot p_\mu,
\end{equation*}
where $s$, $\bm v =(v_1,\cdots,v_r)$ and $\bm p = (p_1,p_2,\cdots)$ are some formal variables,
and $p_\mu = p_{\mu_1} p_{\mu_2}\cdots p_{\mu_l}$
for a partition $\mu=(\mu_1,\cdots,\mu_l)$.
We prove the following (see \S \ref{sec-gdessin} for details):
\begin{Theorem}
There is a cut-and-join type representation
for the generalized dessin partition function $Z_{(r)}(s;v_1,\cdots,v_r;\bm p) $:
\be
\label{eq-cutandjoin-intro}
Z_{(r)}(s;v_1,\cdots,v_r;\bm p)
= \exp\Big( s\cdot \sum_{k=1}^{r+1} a_k(\bm v) P_{-1}^{(k)} \Big) (1).
\ee
Here $a_1(\bm v),\cdots,a_{r+1}(\bm v)$ are the following polynomials in $\bm v = (v_1,\cdots,v_r)$:
\begin{equation*}
a_k (\bm v) =
\frac{(-1)^{k-1}}{k!}e_r(\bm v)
+\frac{1}{k}
\sum_{j=0}^r \Big( \sum_{i=1}^{k-1} \frac{(-1)^{i+k-1} \cdot i^{r-j}}{i!\cdot (k-1-i)!} \Big)  e_j(\bm v),
\end{equation*}
where $e_j(\bm v)$ is the elementary symmetric polynomial of degree $j$ in $v_1,v_2,\cdots,v_r$;
and $P_{-1}^{(k)}$ are the cut-and-join type operators:
\begin{equation*}
P_{-1}^{(k)} = \Res_{z=0} \Big( :z^{k-2}\big( \pd_z + J(z)\big)^{k-1} J(z): \Big)dz,
\end{equation*}
where $J(z) = \sum_{n\in \bZ_+}  n t_n \cdot z^{n-1}
+\sum_{n\in \bZ_+} \frac{\pd}{\pd t_n} \cdot z^{-n-1}$.
\end{Theorem}

In literatures it is known that the above cut-and-join type operators $P_{-1}^{(k)}$
belong to the algebra $W_{1+\infty}$ which gives the symmetry of the KP hierarchy \cite{fkn},
and thus the partition function $Z_{(r)}(s;v_1,\cdots,v_r;\bm p)$
is a tau-functions of the KP hierarchy for every $r$.

When $r=2$,
this recovers the cut-and-join representation
for the dessin partition function derived by Kazarian-Zograf in \cite{kz}.
In that work,
Kazarian and Zograf proved the cut-and-join representation for $Z_{(2)}$
using the combinatorics of ribbon graphs;
and here we prove the above theorem by a different combinatorial method,
namely,
the actions of $W$-type operators on Schur functions derived by Liu and the second-named author \cite{ly},
which enables us to treat $Z_{(r)}$ for general $r\geq 1$ in a unified fashion.
The general $r$ case was also studied in \cite{ammn,ac} using matrix models.
Alexandrov-Mironov-Morozov-Natanzon \cite{ammn}
derived another cut-and-join type representation for this partition function using some infinite Casimir operators.
Here we will need only finitely many $W$-type operators for each $r$.
The $b$-deformed version of the above partition function has been studied by Chapuy-Dołęga in \cite{cd}
using the combinatorics of Jack polynomials,
and they also derived a cut-and-join representation,
see \cite[(62)]{cd}.
(After this work has been completed,
it was pointed out by Dołęga that equation \eqref{eq-cutandjoin-intro} can also be derived
from their cut-and-join representation.)

The rest of this paper is arranged as follows.
In \S \ref{sec-pre} we recall some preliminaries of partitions of integers and Hurwitz numbers.
In \S \ref{sec-mainthm} we give a proof of Theorem \ref{thm-intro-main}.
In \S \ref{sec-gdessin} we apply the main theorem to the generalized dessin counting
to obtain the polynomiality in this case,
and derive the cut-and-join type representation for the generalized dessin partition function.
Finally in \S \ref{sec-onepart} we apply the main theorem to the $m$-part
double Hurwitz numbers with completed cycles.

\section{Preliminaries of Hurwitz Numbers}
\label{sec-pre}

In this section we recall some preliminaries of partitions of integers,
symmetric functions, and Hurwitz numbers.
The materials in this section can be found in \cite{mac, cm}.

\subsection{Partitions of integers and Schur functions}

First we recall the notions of partitions of integers and Schur functions.
For reference, see \cite[\S I]{mac}.

A partition of a nonnegative integer $n$ is a sequence of nonnegative integers
$\lambda=(\lambda_1,\lambda_2,\cdots,\lambda_{l})$
satisfying $\lambda_1\geq \cdots\geq \lambda_{l}> 0$ and $|\lambda|=\lambda_1+\cdots+\lambda_l = n$.
We will denote by $\cP$ the set of all partitions of integers.
In particular,
the empty partition $\lambda = (\emptyset ) \in \cP$ is a partition of $0$.

The number $l$ appearing in a partition $\lambda=(\lambda_1,\cdots,\lambda_{l})$ is called the length of $\lambda$
and will be denoted by $l(\lambda)$.
Let $\lambda = (\lambda_1,\cdots \lambda_l)$ be a partition,
and denote by $m_i (\lambda)$ the number of $i$ appearing in the sequence $(\lambda_1,\cdots \lambda_l)$.
We will use the following notation:
\be\label{eqn:def zmu}
z_\lambda = \prod_{i\geq 1} i^{m_i(\lambda)}\cdot m_i(\lambda)!,
\qquad \forall \lambda\in \cP.
\ee

There is a one-to-one correspondence between the set of partitions of integers
and the set of Young diagrams.
The Young diagram corresponding to $\lambda=(\lambda_1,\cdots,\lambda_{l(\lambda)})$
consists of $|\lambda|$ boxes,
such that there are exactly $\lambda_i$ boxes in the $i$-th row.
Let $\lambda=(\lambda_1,\cdots,\lambda_{l(\lambda)})$ be a partition,
then the transpose $\lambda^t=(\lambda_1^t,\cdots,\lambda_m^t)$ of $\lambda$ is the partition given by
$m=\lambda_1$ and $\lambda_j^t = |\{i| \lambda_i\geq j\}| = \sum_{k\geq j} m_k(\lambda)$.
The Young diagram associated with $\lambda^t$ can be obtained from the Young diagram associated with $\lambda$
by flipping along the main diagonal.
It is clear that $(\lambda^t)^t =\lambda$ for every partition $\lambda$.
The Frobenius notation of a partition $\lambda$ is defined to be:
\begin{equation*}
\lambda= (m_1,m_2,\cdots,m_k | n_1,n_2,\cdots,n_k),
\end{equation*}
where $k$ is the number of boxes on the main diagonal of the Young diagram,
and
\begin{equation*}
m_i = \lambda_i - i, \qquad n_i=\lambda_i^t - i, \qquad i=1,2,\cdots,k.
\end{equation*}

Let $\lambda=(\lambda_1,\cdots,\lambda_{l(\lambda)})$ be a partition of integer.
Let $\square$ be a box in the Young diagram associated with $\lambda$,
and in such a case we will use the notation $\square\in \lambda$.
If $\square$ is at the $i$-th row and $j$-th column of the Young diagram,
then the hook length of the box $\square$ is defined to be:
\be
h(\square) = (\lambda_i-i) +(\lambda_j^t-j) +1,
\ee
and the content $c(\square)$ is defined to be:
\be
c(\square) = j-i.
\ee

Now we recall the definition of the Schur function $s_\lambda$ indexed by a partition $\lambda$.
First consider the special case $\lambda=(m|n) = (m+1,1^n)$.
The Young diagram associated with such a partition is a hook,
and in this case $s_{(m|n)}$ is defined by:
\begin{equation*}
s_{(m|n)}= h_{m+1}e_n - h_{m+2}e_{n-1} + \cdots
+ (-1)^n h_{m+n+1},
\end{equation*}
where $h_n$ and $e_n$ are the
complete symmetric function and elementary symmetric function of degree $n$ respectively.
For a general $\lambda=(m_1,\cdots,m_k|n_1,\cdots,n_k)$,
the Schur function $s_\lambda$ is defined by:
\begin{equation*}
s_\lambda  = \det (s_{(m_i|n_j)} )_{1\leq i,j\leq k}.
\end{equation*}
There are also equivalent definitions for $s_\lambda$:
\begin{equation*}
s_\lambda = \det(h_{\lambda_i-i+j})_{1\leq i,j\leq n}
= \det(e_{\lambda_i^t-i+j})_{1\leq i,j\leq m},
\end{equation*}
for $n\geq l(\lambda)$ and $m\geq l(\lambda^t)$.
The Schur function indexed by the empty partition is defined to be $s_{(\emptyset)} = 1$.
Denote by $\Lambda$ the space of all symmetric functions,
then the set of all Schur functions $\{s_\lambda\}_{\lambda}$ is a basis of  $\Lambda$.

It is well-known that $\{p_\lambda\}$ gives another basis for $\Lambda$,
where $p_\lambda = p_{\lambda_1}\cdots p_{\lambda_l}$ for a partition $\lambda = (\lambda_1,\cdots,\lambda_l)$
and $p_k$ is the Newton symmetric function of degree $k$.
The above two bases for $\Lambda$ are related by (see e.g. \cite{mac}):
\be
\label{eq-newton-schur}
s_\mu = \sum_{\lambda} \frac{\chi^\mu (C_\lambda)}{z_\lambda} p_\lambda,
\ee
where $\chi^\mu$ is the irreducible character of the symmetric group $S_{|\mu|}$ indexed by the partition $\mu$,
and $C_\lambda \subset S_{|\mu|}$ is the conjugacy class whose cycle type is $\lambda$.
The summations on the right-hand sides of \eqref{eq-newton-schur} are over all partitions $\lambda$
with $|\lambda|=|\mu|$.
For an introduction of the representation theory of $S_n$,
see e.g. \cite[\S 4]{fh}.

In some literatures (see e.g. \cite{ly}),
the authors regard $s_\mu$ as a function in the variables $\bm t = (t_1,t_2,\cdots)$ via \eqref{eq-newton-schur}
where $t_n = \frac{p_n}{n}$,
and call it the Schur polynomial indexed by $\mu$.
In fact,
if one defines a degree by setting $\deg (t_n) = n$,
then  $s_\mu (\bm t)$ is a homogeneous polynomial of degree $|\mu|$ in $t_1,t_2,\cdots,t_{|\mu|}$.

\subsection{Hurwitz numbers labeled by $m$ partitions}

First we recall the definition of Hurwitz numbers.
Here we only consider the branched coverings of the Riemann sphere $\bC\bP^1$.

Let $d$ be a positive integer,
and let $\mu^{(1)},\mu^{(2)}\cdots,\mu^{(m)}$ be $m$ partitions of $d$.
Let $ f : \Sigma_g \to \bC\bP^1$ be a branched covering of the Riemann sphere
by a (not necessarily connected) Riemann surface $\Sigma_g$ of genus $g$,
such that there are $m$ branch points $q_1,\cdots,q_m \in \bC\bP^1$,
and the ramification type of $ f $ over $q_i$ is $\mu^{(i)}$.
More precisely,
the preimage of the point $q_i$ consists of $l(\mu^{(i)})$ distinct points $p_1,\cdots, p_{l(\mu^{(i)})}$ on $\Sigma_g$,
such that $f$ locally looks like $z\mapsto z^{\mu^{(i)}_j}$ in some local coordinate $z$ near $p_j$.
The Riemann-Hurwitz formula tells that:
\be
\label{eq-RiemHur}
2g-2 = d(m-2) -\sum_{i=1}^m l(\mu^{(i)}).
\ee
The (disconnected) Hurwitz number $H_{g}^\bullet(\mu^{(1)},\cdots,\mu^{(m)})$ is defined to be
the weighted counting of all such branched coverings $f$:
\begin{equation*}
H_{g}^\bullet(\mu^{(1)},\cdots,\mu^{(m)}) =
\sum_ f  \frac{1}{|\Aut( f )|},
\end{equation*}
where $\Aut( f )$ is the group of automorphisms of the branched covering $f$.
Similarly,
one can define the connected Hurwitz number $H_{g}^\bullet(\mu^{(1)},\cdots,\mu^{(m)})$ by
\begin{equation*}
H_{g}^\circ (\mu^{(1)},\cdots,\mu^{(m)}) =
\sum_{f: \text{ connected}}  \frac{1}{|\Aut( f )|},
\end{equation*}
where the summation is over branched coverings $f$ of ramification type $\mu^{(1)},\cdots,\mu^{(m)}$
from a connected Riemann surface $\Sigma_g$.

\subsection{Group-theoretic description of Hurwitz numbers}

There is a group-theoretic description of Hurwitz numbers,
which leads naturally to the concept of constellations.

Let $ f : \Sigma_g \to \bC\bP^1$ be a branched covering of degree $d$
with $m$ branch points $q_1, q_2, \cdots,q_m \in \bC\bP^1$,
such that the ramification type over $q_i$ is described by a partition $\mu^{(i)}$.
Choose a point $x_0 \in \bC\bP^1\backslash\{q_1,\cdots,q_m\}$,
then the map $f$ induces a group homomorphism
(see e.g. \cite[\S 7]{cm} for details)
\begin{equation*}
\Phi:
\pi_1 (\bC\bP^1 \backslash \{q_1,\cdots,q_m\})
\to S_d,
\end{equation*}
called the monodromy representation,
where $S_d$ is the permutation group of the points in the preimage $f^{-1}(x_0)$ of $x$,
and
\begin{equation*}
\pi_1 (\bC\bP^1 \backslash \{q_1,\cdots,q_m\})
= \langle
\gamma_1,\cdots, \gamma_m |
\gamma_1\cdots\gamma_m = e \rangle.
\end{equation*}
Here $\gamma_i$ denotes the homotopy class of a small loop around $q_i$,
and the cycle type of the permutation $\Phi (\gamma_i)$ is given by the partition $\mu^{(i)}$.
The monodromy representation is called connected if its image $Im(\Phi)$ acts transitively on $\{1,2,\cdots,d\}$.
The connectedness of a branched covering $f$ is equivalent to
the connectedness of the associated monodromy representation.

Conversely,
one can also construct a branched covering $f$ of $\bC\bP^1$ (up to equivalence)
from a given monodromy representation.
Thus one obtains a bijection
\begin{equation*}
\{[f: \Sigma_g \to \bC\bP^1]\} \to
\Hom \big( \pi_1 (\bC\bP^1 \backslash \{q_1,\cdots,q_m\}), S_d \big) / \sim,
\end{equation*}
where $[f]$ denotes the equivalence classes of $f$,
and $\sim$ is the conjugation by elements in $S_d$.
Then the Hurwitz numbers can also be defined algebraically as (see \cite{cm,lz}):
\be
\label{eq-Hurdisconn-gp}
H_{g}^\bullet(\mu^{(1)},\cdots,\mu^{(m)})
= \frac{1}{d!}
\Big| \big\{ ( \alpha_1,\cdots, \alpha_m )\in (S_d)^m \big|
\alpha_1\cdots \alpha_m = e,
\text{ }\alpha_i\in C_i \big\} \Big|,
\ee
where $d=|\mu^{(i)}|$,
and $C_i\subset S_d$ is the conjugacy class whose cycle type is given by $\mu^{(i)}$.
By definition,
the Hurwitz numbers is equal to zero unless $g$ and $\mu^{(i)}$ satisfy the Riemann-Hurwitz formula \eqref{eq-RiemHur}.
And similarly,
\be
H_{g}^\circ(\mu^{(1)},\cdots,\mu^{(m)})
= \frac{1}{d!}
\Bigg| \Big\{  (\alpha_1,\cdots, \alpha_m)   \in (S_d)^m \Big|
\begin{array}{l}
\alpha_1\cdots \alpha_m = e, \text{ }
\alpha_i\in C_i, \text{ and}  \\
 \text{$\langle \alpha_1,\cdots,\alpha_m \rangle$ acts transitively}
\end{array}
\Big\} \Bigg|.
\ee

This construction leads naturally to the notion of constellations, see e.g. \cite[\S 1]{lz} for reference.
A $k$-constellation is a sequence $[\sigma_1,\cdots,\sigma_k]$ of permutations $\sigma_i \in S_d$
such that:
\begin{itemize}
\item[1)]
$\sigma_1 \sigma_2\cdots \sigma_k = e \in S_d$;
\item[2)]
the subgroup $\langle \sigma_1,\cdots, \sigma_k \rangle$
acts transitively on the set $\{1,\cdots,d\}$.
\end{itemize}
Here $k$ and $d$ are called the length and degree of this constellation respectively.
Then the computation of connected Hurwitz numbers of genus $g$ and ramification type $(\mu^{(1)},\cdots,\mu^{(k)})$,
is actually equivalent to the enumeration of constellations $[\sigma_1,\cdots,\sigma_k]$ of length $k$ and degree $d$,
such that the cycle type of each $\sigma_i$ is exactly $\mu^{(i)}$.
The genus $g$ is determined from $\mu^{(i)}$ by the Riemann-Hurwitz formula \eqref{eq-RiemHur}.

\subsection{Burnside formula for Hurwitz numbers}

One can compute the disconnected Hurwitz numbers algebraically
using the above group-theoretic description.
Let $C_i\subset S_d$ be the conjugacy class in $S_d$ whose cycle type is given by the partition $\mu^{(i)}$,
then the following famous result is called the Burnside character formula
(see e.g. \cite[\S 9]{cm} for a proof):
\be
\label{eq-Burnside}
H_g^\bullet (\mu^{(1)},\cdots,\mu^{(m)})
= \sum_{|\eta| =d} \Big(\frac{\dim V^\eta}{d!}\Big)^2 \cdot
\prod_{i=1}^m \frac{|C_i|\cdot \chi^\eta (C_i)}{\dim V^\eta},
\ee
where
$V^\eta$ denotes the irreducible representation of $S_d$ indexed by the partition $\eta$,
and $\chi^\eta$ is the irreducible character associated with $V^\eta$.
The summation in the right-hand side is over all partitions $\eta$ of $d$.

The numbers $|C_i|$, $\chi^\eta (C_i)$, and $\dim V^\eta$ appearing in the right-hand side
of the above formula can all be represented combinatorially in terms of the partitions $\eta$ and $\mu^{(i)}$.
The size of the conjugacy class $C_i$ is:
\be
\label{eq-conjclasssize}
|C_i| = \frac{d!}{\prod_{k\geq 1} k^{m_k}\cdot m_k(\mu^{(i)})!}= \frac{d!}{z_{\mu^{(i)}}},
\ee
and the dimension of $V^\eta$ is given by the hook length formula:
\be
\label{eq-dim-hook}
\dim V^\eta = \frac{d!}{ \prod_{\square\in \eta} h(\square)},
\ee
where $\square\in \eta$ means $\square$ runs over boxes in the Young diagram associated with $\eta$,
and $h(\square)$ is the hook length of the box $\square$.
The evaluation of the character $\chi^\eta (C_i)$ can be computed using the Frobenius formula,
and for details we refer to \cite[\S 4]{fh}.

\section{$\bN$-Coefficient Binomial Polynomiality of Hurwitz Numbers}
\label{sec-mainthm}

In this section we describe and prove the $\bN$-coefficient binomial polynomiality
of Hurwitz numbers and then give some examples.
The proof is inspired by the work \cite{st} of Stanley.

\subsection{$\bN$-Coefficient binomial polynomiality}
\label{sec-Npol-H-pf}

For a nonnegative integer $d$,
we denote by $\cP_d$ the set of all partitions of $d$.
Now let $r$ be a positive integer,
and let $\lambda^1,\lambda^2,\cdots,\lambda^r$ and $\mu$ be some partitions of integers
such that $|\mu|\geq |\lambda^i|$ for every $i$.
For every positive integer $k$,
denote:
\be
H_{\lambda^1,\cdots,\lambda^r;k}^\circ (\mu)
= \sum_{\substack{ \sigma\in \cP_{|\mu|} \\ \text{s.t. } l(\sigma) = k }} H_{g}^\circ
(\tilde\lambda^1,\cdots,\tilde\lambda^r,\sigma,\mu),
\ee
where $\tilde\lambda^i$ is the partition
$(\lambda^i,1^{|\mu|-|\lambda^i|}) = (\lambda^i_1,\cdots,\lambda^i_{l(\lambda^i)},1,\cdots,1)$ of $|\mu|$.
The genus $g$ in the right-hand side is uniquely determined by the Riemann-Hurwitz formula:
\be
\begin{split}
2g-2 =
|\mu|\cdot r - \sum_{i=1}^r l(\tilde\lambda^i) - l(\sigma) -l(\mu)
= \sum_{i=1}^r \big(|\lambda^i| - l(\lambda^i)\big) -k- l(\mu).
\end{split}
\ee
Let $z_\mu$ be the number defined by \eqref{eqn:def zmu},
and we have:
\begin{Theorem}
\label{thm-binompol-H}
Fix three positive integers $r,k,l>0$ and $r$ partitions of integers $\lambda^1,\cdots,\lambda^r$.
Let $\mu = (\mu_1,\cdots,\mu_l)$ be a partition of length $l$ with $|\mu|\geq |\lambda^i|$ for every $i$,
then
\be
z_\mu \cdot H_{\lambda^1,\cdots,\lambda^r;k}^\circ (\mu)
\ee
is a linear combination of
\be
\binom{\mu_1}{d_1} \binom{\mu_2}{d_2} \cdots \binom{\mu_l}{d_l}
\ee
with non-negative integer coefficients,
where $d_j\geq 1$ for every $1\leq j\leq l$, and
\be
\label{eq-deg-ineq}
\sum_{j=1}^l d_j \leq \sum_{i=1}^r 2 \big( |\lambda^i| - l(\lambda^i)\big).
\ee
\end{Theorem}
\begin{proof}
Denote $n= |\mu|$.
For a partition $\lambda$ of $n$,
we denote by $C_\lambda \subset S_n$ the conjugation class of permutations
whose cycle type is $\lambda$.
And for a permutation $\alpha\in S_n$,
denote by $l(\alpha)$ the number cycles in $\alpha$
(i.e., the length of the cycle type of $\alpha$).
The group-theoretic description of Hurwitz numbers tells that:
\begin{equation*}
n! \cdot  H_{\lambda^1,\cdots,\lambda^r;k}^\circ (\mu)
= \Bigg| \Bigg\{ (\alpha_1,\cdots,\alpha_r,\beta,\gamma)\in  (S_n)^{r+2} \Bigg|
\begin{array}{l}
\alpha_i \in C_{\tilde\lambda^i}, \text{ } l(\beta) = k, \text{ } \gamma\in C_\mu, \\
 \alpha_1\alpha_2\cdots\alpha_r\beta\gamma = e, \text{ and}\\
\text{$\langle \alpha_1,\cdots,\alpha_r,\beta,\gamma \rangle$ acts transitively}
\end{array}
\Bigg\} \Bigg|.
\end{equation*}
Recall that
$z_\mu = n!/|C_\mu|$ is the order of the centralizer of an element in $C_\mu$.
Moreover,
the order of the set in the right-hand side
is independent of the choice of $\gamma \in C_\mu$,
and thus one can fix $\gamma \in C_\mu$ to be:
\be
\label{eq-def-gamma0}
\gamma_0 = (1,2,\cdots,\mu_1)(\mu_1+1,\cdots,\mu_1+\mu_2)\cdots
(\sum_{j=1}^{l-1}\mu_j +1 ,\cdots, \sum_{j=1}^l \mu_j),
\ee
and then:
\be
\label{eq-Hbinompol-pf1}
\begin{split}
&z_\mu \cdot H_{\lambda^1,\cdots,\lambda^r;k}^\circ (\mu) \\
=& \Bigg|
\Big\{ (\alpha_1,\cdots,\alpha_r,\beta) \in  (S_n)^{r+1} \Big|
\begin{array}{l}
\alpha_i \in C_{\tilde\lambda^i}, \text{ } l(\beta) = k,  \text{ }
 \alpha_1\cdots\alpha_r\beta\gamma_0 = e, \\
\text{and $\langle \alpha_1,\cdots,\alpha_r,\beta,\gamma_0 \rangle$ acts transitively}
\end{array}
\Big\} \Bigg|.
\end{split}
\ee
Here $\beta$ is uniquely determined by the choice of $\alpha_1,\cdots,\alpha_r$:
\begin{equation*}
\beta = \alpha_r^{-1} \alpha_{r-1}^{-1} \cdots \alpha_1^{-1} \gamma_0^{-1}.
\end{equation*}
Now denote $[n] = \{1,2,\cdots,n\}$,
and for $1\leq i\leq r$ define:
\be
F_i = \big\{ j\in [n] \big| \text{$j$ is not fixed by $\alpha_i:[n]\to [n]$} \big\},
\ee
and define
\be
F = \bigcup_{i=1}^r F_i \subset [n].
\ee
By the definition of $F$ we know that $[n]\backslash F \subset [n]$ is the subset of elements
fixed by every $\alpha_i$.
Notice that $n-|F_i|$ is the number of $1$-cycles in $\alpha_i$,
and $\alpha_i \in C_{\tilde\lambda^i} \subset S_n$ contains
\begin{equation*}
l(\alpha_i) = l(\tilde\lambda^i) = l(\lambda^i) + n -|\lambda^i|
\end{equation*}
cycles in total,
and thus it contains $\big(l(\alpha_i) - \big(n-|F_i|)\big)$ cycles of size strictly greater than $1$.
Then one has:
\be
\label{eq-ineq-deg-pf}
n \geq 1\cdot (n-|F_i|) +2\cdot \big(l(\alpha_i) - \big(n-|F_i|)\big),
\ee
which implies:
\begin{equation*}
|F_i| \leq 2|\lambda^i| -2 l(\lambda^i).
\end{equation*}
And then the size of $F$ has an upper bound:
\be
\label{eq-upperbd-pf1}
|F| \leq \sum_{i=1}^r |F_i|
\leq \sum_{i=1}^r 2\big(|\lambda^i| - l(\lambda^i)\big).
\ee

For $1\leq s \leq l$,
define
\be
\label{eq-def-Gammas}
\Gamma_s = \Big\{ \sum_{j=1}^{s-1} \mu_j +1, \sum_{j=1}^{s-1} \mu_j +2, \cdots, \sum_{j=1}^s \mu_j \Big\}
\subset [n]
\ee
to be the numbers appearing in the $s$-th cycle in $\gamma_0$.
Now for a choice of permutations $\alpha_1,\cdots,\alpha_r\in S_n$ satisfying
the conditions in the right-hand side of \eqref{eq-Hbinompol-pf1},
we have:
\begin{equation*}
F = (F\cap \Gamma_1) \sqcup (F\cap \Gamma_2) \sqcup \cdots \sqcup (F\cap \Gamma_l).
\end{equation*}
Moreover, one has
\begin{equation*}
F\cap \Gamma_s \not= \emptyset, \qquad \forall 1\leq s\leq l,
\end{equation*}
since the action of $\langle \alpha_1,\cdots, \alpha_r, \gamma_0\rangle$ on $[n]$ is transitive.

Now let $d_1, d_2,\cdots,d_l$ be some integers with $1\leq d_s \leq \mu_s$.
We fix $d_s$ distinct numbers
\begin{equation*}
b_{s,1},b_{s,2},\cdots, b_{s,d_s} \in
\Gamma_s,
\end{equation*}
and denote $\bm b_s = \{ b_{s,1},b_{s,2},\cdots, b_{s,d_s}\}$.
Denote:
\begin{equation}
C_{\bm b_1,\cdots, \bm b_l} =
\Bigg\{ (\alpha_1,\cdots,\alpha_r) \in  (S_n)^r \Bigg|
\begin{array}{l}
\alpha_i \in C_{\tilde\lambda^i},
\text{ } l(\alpha_r^{-1} \alpha_{r-1}^{-1} \cdots \alpha_1^{-1} \gamma_0^{-1}) = k, \\
\text{$\langle \alpha_1,\cdots,\alpha_r,\gamma_0 \rangle$ acts transitively on $[n]$,} \\
\text{and $F\cap \Gamma_s = \bm b_s$ for every $s$}
\end{array}
\Bigg\},
\end{equation}
where $F\subset [n]$ is the subset of elements which are not fixed by at least one of
these $\alpha_1,\cdots,\alpha_r$.
Then by \eqref{eq-Hbinompol-pf1} we have:
\begin{equation*}
z_\mu \cdot H_{\lambda^1,\cdots,\lambda^r;k}^\circ (\mu) =
\sum_{d_1,\cdots,d_l \geq 1}
\sum_{\substack{\bm b_s \subset \Gamma_s\\ \text{s.t. }|\bm b_s| = d_s}}
|C_{\bm b_1,\cdots, \bm b_l}|.
\end{equation*}
We now claim:
\begin{itemize}
\item[1)]
For fixed numbers $d_1\cdots,d_l\geq 1$,
the size $|C_{\bm b_1,\cdots, \bm b_l}|$ is independent of the choice of the subsets $\bm b_s \subset \Gamma_s$;
\item[2)]
the size $|C_{\bm b_1,\cdots, \bm b_l}|$ is independent of the choice of $\mu=(\mu_1,\cdots,\mu_l)$.
\end{itemize}
If the two claims are true,
then we simply denote the size $|C_{\bm b_1,\cdots, \bm b_l}|$ by $c_{d_1,\cdots,d_l} \in \bN$
which is independent of $\mu$,
and then:
\begin{equation}
z_\mu \cdot  H_{\lambda^1,\cdots,\lambda^r;k}^\circ (\mu) =
\sum_{d_1,\cdots,d_s \geq 1}
c_{d_1,\cdots,d_s} \cdot \binom{\mu_1}{d_1}\binom{\mu_2}{d_2}\cdots \binom{\mu_l}{d_l}.
\end{equation}
Here $d_1+\cdots d_l = |F| \leq  \sum_{i=1}^r 2\big(|\lambda^i| - l(\lambda^i)\big) $
by \eqref{eq-upperbd-pf1},
and in this way we obtain the conclusion of the theorem.
In what follows we prove the two claims.

For simplicity we first consider the case $r=1$ and $l=1$.
Let $\mu = (n)$ and $\alpha\in S_n$,
and let $F\subset [n]$ be the elements not fixed by $\alpha$,
then the restriction $\alpha|_F :F\to F$ is a permutation of the subset $F$.
Notice that if a number $p\in [n]$ is not in $F$,
then one always has
\begin{equation*}
\big( \alpha^{-1}\cdot (n,n-1,\cdots,1) \big) (p+1) = p,
\end{equation*}
and thus $p$ and $p+1$ are in the same cycle and $p+1$ is followed by $p$.
In this case we simply package $k$ and $p+1$ together;
and if $p+1$ is still not in $F$, then we package $p,p+1$ and $p+2$ together.
Repeat this procedure (and we use the convention $n+1=1$),
and we partition the set $[n]$ into some subsets such that
these subsets are in one-to-one correspondence with $F$.
Such a subset is always contained in
a cycle of the product $\alpha^{-1}\cdot (n,n-1,\cdots,1)$
(and a cycle may contain several such subsets).
Then one can simply rename such a subset $\{p,p+1,\cdots,p+s\}$ as $\overline{p+s}$,
and in this way one sees that the result is actually independent of
the choice of $n$ and the choice of $F\subset[n]$
as long as the size $|F|$ is fixed.

For example,
let $\alpha = (153)$,
and then:
\begin{equation*}
(135)^{-1}\cdot (n,n-1,\cdots,1) = (n,n-1,\cdots,6,3,2,5,4,1),
\end{equation*}
Here $2$ is fixed by $\alpha$ and thus
in the product $3$ is followed by $2$,
and thus we package $2,3$ together and denote this package by $\overline{3}$.
Similarly,
we package $4,5$ together and denote this package by $\overline{5}$;
and package $6,7,\cdots,n,1$ together and denote this package by $\overline{1}$.
Then the above product can be represented as:
\begin{equation*}
(\overline{1},\overline{3},\overline{5})^{-1}\cdot (\overline{1},\overline{5},\overline{3})
= (\overline{1}, \overline{3}, \overline{5}).
\end{equation*}
If one is only concerned with the number of cycles in the product,
then one can further rename $\overline{1},\overline{3},\overline{5}$
as $1,2,3$ respectively
and the above equality is actually equivalent to the computation $(123)^{-1}\cdot(132) = (123)$.
In this way one reduces the computations for a general $n$ and $F=\{1,3,5\}\subset[n]$
to the simpler case $n=3$ and $F=\{1,2,3\}$.

For a general $r>0$ and $l>0$,
the proof is essentially the same as the above special case.
In fact,
one needs to consider the cycles within each subset $\Gamma_s \subset [n]$
and partition each cycle into some subsets
by letting the elements in $[n]\backslash F$ be absorbed into elements in $F$
where $[n]\backslash F$ is the subset of elements fixed by each $\alpha_i$.
Here we do not repeat the details.
\end{proof}

\begin{Remark}
If there exists an $i$ such that $\lambda^i$ contains a cycle of size greater than $2$,
then the inequality \eqref{eq-ineq-deg-pf} is strict
and thus the degree bound inequality \eqref{eq-deg-ineq} is strict.
\end{Remark}

\begin{Remark}
\label{rmk-disconnpol}
The same method can be applied to prove
the polynomiality of the disconnected Hurwitz numbers $H_{\lambda^1,\cdots,\lambda^r;k}^\bullet (\mu)$.
One can simply delete the transitivity condition in the
group-theoretic descriptions above
and obtain the proof of the disconnected case.
\end{Remark}

As a straightforward consequence of the above theorem,
we have:
\begin{Corollary}
Let three positive integers $r,k,l>0$ and $r$ partitions $\lambda^1, \lambda^2,\cdots,\lambda^r$ be fixed.
Then for a partition $\mu = (\mu_1,\cdots,\mu_l)$ of length $l$ with $|\mu| \geq |\lambda^i|$ for every $1\leq i\leq r$,
the Hurwitz number $z_\mu \cdot H_{\lambda^1,\cdots,\lambda^r;k}^\circ (\mu)$ defined above is a polynomial
with rational coefficients in $\mu_1,\mu_2,\cdots,\mu_l$
of degree less than or equals to $\sum_{i=1}^r 2 \big( |\lambda^i| - l(\lambda^i)\big)$.
\end{Corollary}

\subsection{Examples}

Here we give some examples of the $H_{\lambda^1,\cdots,\lambda^r;k}^\circ (\mu)$
for small $l$ and $\lambda^i$.

\begin{Example}
First consider the case $l=1$,
and we denote $\mu = (n)$.
In the simplest nontrivial case $\lambda = (2)$,
for every $k\geq 1$ we have:
\begin{equation*}
n\cdot H_{(2);k}^\circ \big((n)\big) = \sum_{d=1}^2 c_{k,d}^{(2)}\cdot \binom{n}{d},
\end{equation*}
where:
\begin{equation*}
c_{k,d}^{(2)} =
\Big|
\big\{ \alpha \in S_n \big|
\alpha \in C_{(2,1,\cdots,1)},
\text{ } l\big(\alpha^{-1}\cdot (n,n-1,\cdots,1)\big) = k,
\text{ } F = \{1,2,\cdots,d\}
\big\}
\Big|.
\end{equation*}
where $F\subset [n]$ is the subset of elements not fixed by $\alpha$.
Here $\alpha$ is a transposition and thus $c_{k,d}^{(2)} = 0$ unless $d=2$,
and then one has $\alpha = (12)$.
Notice that
\begin{equation*}
(12)^{-1}\cdot (n,n-1,\cdots,1)
= (n,n-1,\cdots,3,1)(2),
\end{equation*}
and thus $c_{k,2}^{(2)} = \delta_{k,2}$,
and we conclude that
\begin{equation*}
n\cdot H_{(2);k}^\circ \big((n)\big) = \delta_{k,2} \cdot \binom{n}{2}
= \delta_{k,2}\cdot \half n(n-1).
\end{equation*}

For $\lambda = (3)$ we have:
\begin{equation*}
n\cdot H_{(3);k}^\circ \big((n)\big) = \sum_{d=1}^4 c_{k,d}^{(3)}\cdot \binom{n}{d},
\end{equation*}
where:
\begin{equation*}
c_{k,d}^{(3)} =
\Big|
\big\{ \alpha \in S_n \big|
\alpha \in C_{(3,1,\cdots,1)},
\text{ } l\big(\alpha^{-1}\cdot (n,n-1,\cdots,1)\big) = k,
\text{ } F = \{1,2,\cdots,d\}
\big\}
\Big|.
\end{equation*}
Here $c_{k,d}^{(3)} = 0$ unless $d=3$ since $\alpha$ is a $3$-cycle.
Notice that:
\begin{equation*}
\begin{split}
&(123)^{-1} \cdot (n,n-1,\cdots,1) = (n,n-1,\cdots,4,2,3,1),\\
&(132)^{-1} \cdot (n,n-1,\cdots,1) = (n,n-1,\cdots,4,1)(2)(3),
\end{split}
\end{equation*}
and thus $c_{k,3}^{(3)} = \delta_{k,1} + \delta_{k,3}$,
and we conclude that:
\begin{equation*}
n\cdot H_{(3);k}^\circ \big((n)\big) = (\delta_{k,1} + \delta_{k,3}) \cdot \binom{n}{3}.
\end{equation*}

Similarly for $\lambda = (2,2)$, one has $d=4$. And by
\begin{equation*}
\begin{split}
&\big((12)(34)\big)^{-1} \cdot (n,n-1,\cdots,1) = (n,n-1,\cdots,5,3,1)(4)(2),\\
&\big((13)(24)\big)^{-1} \cdot (n,n-1,\cdots,1) = (n,n-1,\cdots,5,2,3,4,1),\\
&\big((14)(23)\big)^{-1} \cdot (n,n-1,\cdots,1) = (n,n-1,\cdots,5,1)(24)(3),\\
\end{split}
\end{equation*}
one concludes that:
\begin{equation*}
n\cdot H_{(2,2);k}^\circ \big((n)\big) = (\delta_{k,1} + 2 \delta_{k,3}) \cdot \binom{n}{4}.
\end{equation*}
\end{Example}

\begin{Example}
We give an example for $l=2$.
Let $\mu = (\mu_1,\mu_2)$ and $\lambda = (2)$,
then:
\begin{equation*}
z_\mu\cdot H_{(2);k}^\circ \big((\mu_1,\mu_2)\big)
= \sum_{d_1,d_2\geq 1} c_{k;d_1,d_2}^{(2)}
\cdot \binom{\mu_1}{d_1}\binom{\mu_2}{d_2}.
\end{equation*}
Here $d_1+d_2\leq 2(|\lambda| - l(\lambda))=2$,
and thus $d_1=d_2=1$.
Denote $n=\mu_1+\mu_2$,
and one has:
\begin{equation*}
c_{k;1,1}^{(2)} =
\Bigg|
\Bigg\{ \alpha \in S_n \Bigg|
\begin{array}{l}
\alpha \in C_{(2,1,\cdots,1)},
\text{ } l\big(\alpha^{-1}\cdot \gamma_0^{-1}\big) = k, \\
F\cap \Gamma_1 = \{1\},
\text{ } F\cap \Gamma_2 = \{\mu_1+1\}, \\
\text{and $\langle \alpha,\gamma_0 \rangle$ acts transitively on $[n]$}
\end{array}
\Bigg\}
\Bigg|,
\end{equation*}
where $\gamma_0 = (1,2,\cdots,\mu_1)(\mu_1+1,\mu_1+2,\cdots, n)$,
and
\begin{equation*}
\Gamma_1 = \{1,2,\cdots, \mu_1\},
\qquad
\Gamma_2 = \{\mu_1+1,\mu_1+2,\cdots, n\}.
\end{equation*}
Therefore $\alpha$ must be the transposition $(1,\mu_1+1)$.
Notice that
\begin{equation*}
(1,\mu_1+1)^{-1}\cdot \gamma_0^{-1} =
(\mu_1,\mu_1-1,\cdots,2,\mu_1+1,n,n-1,\cdots,\mu_1+2,1),
\end{equation*}
and thus $c_{k;1,1}^{(2)} = \delta_{k,1}$ and
\begin{equation*}
z_\mu\cdot H_{(2);k}^\circ \big((\mu_1,\mu_2)\big)
= \delta_{k,1}\cdot \binom{\mu_1}{1}\binom{\mu_2}{1}
= \delta_{k,1}\cdot \mu_1\mu_2.
\end{equation*}
\end{Example}

\begin{Example}
\label{eg-3.3pol}
Here we consider an example for $l=1$ and $r=2$.
Let $\lambda^1 = \lambda^2 =(2)$, and denote $\mu = (n)$.
Then:
\begin{equation*}
n\cdot H_{(2),(2);k}^\circ \big((n)\big) = \sum_{d=1}^4 c_{k,d}^{(2),(2)}\cdot \binom{n}{d},
\end{equation*}
where
\begin{equation*}
c_{k,d}^{(2),(2)} =
\Bigg|
\Bigg\{ \alpha_1,\alpha_2 \in S_n \Bigg|
\begin{array}{l}
\alpha_1,\alpha_2 \in C_{(2,1,\cdots,1)},
\text{ } F=\{1,2,\cdots,d\}, \\
\text{and } l\big(\alpha_2^{-1} \alpha_1^{-1} \cdot (n,n-1,\cdots,1)\big) = k
\end{array}
\Bigg\}
\Bigg|,
\end{equation*}
where $F\subset [n]$ is the subset of elements not fixed by $\alpha_1$ or $\alpha_2$.
Here $\alpha_1$, $\alpha_2$ are two transpositions and thus $2\leq d\leq 4$.
\begin{itemize}
\item[1)]
If $d=2$,
then $\alpha_1 = \alpha_2  = (12)$,
and then $k=l\big((n,n-1,\cdots,1)\big)=1$.
Thus the case $d=2$ contributes a term $\delta_{k,1} \cdot \binom{n}{2}$ to
$n\cdot H_{(2),(2);k}^\circ \big((n)\big)$.

\item[2)]
If $d=3$,
there are $6$ choices of the pair $(\alpha_1,\alpha_2)$,
and in $3$ cases $\alpha_2^{-1} \alpha_1^{-1} = (123)$
while in other $3$ cases $\alpha_2^{-1} \alpha_1^{-1} = (132)$.
Thus the case $d=3$ contributes $(3\delta_{k,1}+3\delta_{k,3}) \cdot \binom{n}{3}$.

\item[2)]
If $d=4$,
$\alpha_2^{-1} \alpha_1^{-1} = (12)(34)$, $(13)(24)$, or $(14)(23)$,
and in each case there are two possible choices of the pair $(\alpha_1,\alpha_2)$.
Thus the case $d=4$ contributes $(2\delta_{k,1}+4\delta_{k,3}) \cdot \binom{n}{4}$.
\end{itemize}
In summary,
one has:
\begin{equation*}
n\cdot H_{(2),(2);k}^\circ \big((n)\big) =
\delta_{k,1} \Big(\binom{n}{2} + 3\binom{n}{3} + 2\binom{n}{4}\Big)
+ \delta_{k,3} \Big( 3\binom{n}{3} + 4\binom{n}{4} \Big).
\end{equation*}
\end{Example}

\section{Application: Polynomiality of Generalized Dessin Counting}
\label{sec-gdessin}

In this section we apply the above main theorem
to the case of the generalized dessin counting,
and obtain the $\bN$-coefficient binomial polynomiality in this case.
We also present a cut-and-join type representation
for the generalized dessin partition function.

\subsection{Generalized Dessin Counting}

First
we briefly review the enumeration of Grothendieck's dessins d'enfants
and its generalization.

Let $X$ be a smooth algebraic curve of genus $g$ defined over $\bC$.
A famous result of Bely\u{\i} \cite{be} tells that
$X$ can be defined over the field of algebraic numbers $\overline\bQ$ if and only if
there exists a holomorphic map $f:X\to\bC\bP^1$ such that
$f$ is ramified only over three points $0,1,\infty \in \bC\bP^1$.
Here $(X,f)$ is called a Bely\u{\i} pair if such a branched covering $f$ exists.
Given a Bely\u{\i} pair $(X,f)$,
consider the preimage $f^{-1}([0,1])$ of the segment $[0,1]\subset \bC\bP^1$ in the curve $X$.
Denote by $d$ the degree of the branched covering $f$,
then the preimage $f^{-1}([0,1])$ is a ribbon graph of genus $g$ with $d$ edges,
whose vertices are bicolored such that adjacent vertices have different colors.
Here the vertices of this ribbon graph are actually the preimage of the two points $0$ and $1$,
and the vertices in $f^{-1}(0)$ and $f^{-1}(1)$ are distinguished by two colors.
In \cite{gr},
Grothendieck shows that isomorphism classes of Bely\u{\i} pairs $(X,f)$ are
in one-to-one correspondence with such connected bicolored ribbon graphs,
and these connected bicolored ribbon graphs are called Grothendieck's dessins d'enfants.

The enumeration of Grothendieck's dessins d'enfants has been discussed
by many authors in literatures of mathematical physics,
see e.g. \cite{ammn, ac, dmss, kz, kls, yz, zhou2, zhou3, zog}.
Let $\mu = (\mu_1,\cdots,\mu_m)$ be a partition of integer,
and consider the weighted counting of dessins:
\begin{equation*}
N_{k,l} (\mu) = \sum_{\Gamma\in \cD_{k,l;\mu}} \frac{1}{|\Aut(\Gamma)|},
\end{equation*}
where $k,l$ are two positive integers,
and $\cD_{k,l;\mu}$ is the set of dessins whose associated Bely\u{\i} pair $(X,f)$
is of degree $d=|\mu|$ and satisfies:
\begin{itemize}
\item[1)]
$|f^{-1}(0)|=k$, and $|f^{-1}(1)|=l$;
\item[2)]
The ramification type over $\infty \in \bC\bP^1$ is described by the partition $\mu$,
\end{itemize}
and $\Aut(\Gamma)$ denotes the group of automorphisms of the dessin $\Gamma$ that preserves each boundary component.
The above dessin counting number $N_{k,l} (\mu)$
is actually a summation of Hurwitz numbers which count the branched coverings of $\bC\bP^1$ ramified over three points,
such that the preimages of the first two points consist of $k$ and $l$ elements respectively,
and the ramification type over the third point is described by $\mu$.
In \cite{kz},
Kazarian and Zograf derived the Virasoro constraints, Eynard-Orantin topological recursion,
and a cut-and-join equation for this dessin counting problem using combinatorial properties of the dessins d'enfants,
and proved that the generating series of $N_{k,l} (\mu)$ is the logarithm of a tau-function
of the KP hierarchy.
The KP-affine coordinates and emergent geometry of this tau-function
has been discussed by Zhou in \cite{zhou2, zhou3}.

It is natural to generalize the above dessin counting problem to the case
where the number of branch points on $\bC\bP^1$ is not necessarily to be three.
Let $\mu = (\mu_1,\cdots,\mu_l)$ be a partition of integer,
and let $1\leq k_1,\cdots,k_r$ be a sequence of positive integers (where $r\geq 1$).
Consider the branched coverings $f:\Sigma_g \to \bC\bP^1$ from a connected Riemann surface to $\bC\bP^1$
ramified over $r+1$ points
$q_1,\cdots,q_r,q_{r+1} \in \bC\bP^1$,
such that the preimage $f^{-1}(q_i)$ consists of $k_i$ elements for $1\leq i\leq r$,
and the ramification type over $q_{r+1}$ is described by $\mu$.
We denote by
\be
N_{k_1,\cdots,k_r}^\circ (\mu) = \sum_{f: \text{ connected}} \frac{1}{|\Aut(f)|}
\ee
the weighted counting of all such branched coverings.
Similarly,
we denote by
\be
N_{k_1,\cdots,k_r}^\bullet (\mu) = \sum_{f} \frac{1}{|\Aut(f)|}
\ee
the weighted counting of branched coverings of this ramification type
where we allow the Riemann surface $\Sigma_g$ to be disconnected.
It is clear that $N_{k_1,\cdots,k_r}^\bullet (\mu) = N_{k_1,\cdots,k_r}^\circ (\mu) =0$
unless $k_i\leq|\mu|$ for every $i$.
In the special case $r=2$,
the weighted counting $N_{k_1,k_2}^\circ(\mu)$ recovers
the dessin counting $N_{k_1,k_2}(\mu)$ studied in \cite{kz}.

\subsection{Schur function expansion for the generalized dessin partition function}

This subsection and the next subsection
are devoted to the properties of the generating series of
the (disconnected) generalized dessin counting.
These properties will not be used in the discussions of the polynomiality,
but here we describe them for completeness.
In this subsection we represent the partition functions
in terms of Schur functions.

Consider the following generating series of the
generalized dessin counting:
\be
Z_{(r)}(s;v_1,\cdots,v_r;\bm p)
= \sum_{k_1,\cdots,k_r = 1}^\infty \sum_{\mu\in \cP}
\Big(\prod_{i=1}^r v_i^{k_i}\Big) s^{|\mu|} N_{k_1,\cdots,k_r}^\bullet (\mu) \cdot p_\mu,
\ee
where $s,v_1,\cdots,v_r$ and $\bm p = (p_1,p_2,\cdots)$ are some formal variables,
and
$p_\mu = p_{\mu_1} p_{\mu_2}\cdots p_{\mu_l}$
for a partition $\mu=(\mu_1,\mu_2,\cdots,\mu_l)$.
We call it the generalized dessin partition function.
For $r=2$,
the partition function $Z_{(2)}(s;v_1,v_2;\bm p)$ coincide with the dessin partition function
studied in \cite{kz}.
By the Burnside character formula \eqref{eq-Burnside} and \eqref{eq-conjclasssize} one has:
\begin{equation*}
N_{k_1,\cdots,k_r}^\bullet(\mu)
= \sum_{\substack{\mu^{(1)},\cdots,\mu^{(r)} \in \cP_d \\ \text{s.t. }l(\mu^{(i)}) = k_i}} \sum_{\eta\in \cP_d}
\Big(\frac{d!}{\dim V^\eta} \Big)^{r-1} \cdot \frac{ \chi^\eta (C)}{z_\mu }
\prod_{i=1}^{r} \frac{ \chi^\eta (C_i)}{z_{\mu^{(i)}}},
\end{equation*}
where $C_i\subset S_d$ (resp. $C\subset S_d$) is the conjugacy class whose cycle type is $\mu^{(i)}$ (resp. $\mu$).
And then by \eqref{eq-newton-schur} we have:
\begin{equation*}
Z_{(r)}(s;v_1,\cdots,v_r;\bm p)
= \sum_{d=0}^\infty
\sum_{ \eta,\mu^{(1)},\cdots,\mu^{(r)}\in \cP_d}
\Big(\frac{d!}{\dim V^\eta} \Big)^{r-1} \cdot
 s^{d}\cdot s_\eta (\bm t)
\prod_{i=1}^{r} \frac{ \chi^\eta (C_i) \cdot v_i^{l(\mu^{(i)})} }{z_{\mu^{(i)}}},
\end{equation*}
where $\bm t=(t_1,t_2,t_3,\cdots)$, and
$t_n = \frac{p_n}{n}$ for every $n\geq 1$.
For each fixed $i$ ($1\leq i\leq r$),
taking $p_k = v_i$ for every $k\geq 1$ in \eqref{eq-newton-schur} gives:
\begin{equation*}
\sum_{\mu^{(i)}\in \cP} \frac{ \chi^\eta (C_i) \cdot v_i^{l(\mu^{(i)})} }{z_{\mu^{(i)}}}
= s_\eta (\bm t)\big|_{p_1 = p_2 = p_3 =\cdots  = v_i}
= s_\eta (t_k = \frac{v_i}{k}),
\end{equation*}
and thus we have:
\be
\label{eq-Z-Schurexp-1}
 Z_{(r)}(s;v_1,\cdots,v_r;\bm p)
=
\sum_{ \eta\in \cP}
\frac{s^{|\eta|} \cdot (|\eta|!)^{r-1}}{(\dim V^\eta)^{r-1}}
 \cdot s_\eta (\bm t)
\prod_{i=1}^{r} s_\eta (t_k = \frac{v_i}{k}).
\ee
The following evaluation formula
is well-known in literatures  (see e.g. \cite[\S I.3. Example 4]{mac};
and we will give a new proof in Appendix \ref{app-pflem5.1}) based on the $W$-action on Schur functions \cite{ly}:
\begin{Lemma}
\label{lem-eval-schur}
Take $t_k = \frac{v}{k}$ for every $k\geq 1$,
then one has:
\be
s_\mu (t_k=\frac{v}{k}) =
\prod_{\square \in\mu} \frac{v+c(\square)}{h(\square)},
\ee
where $c(\square)$ and $h(\square)$ are the content and hook-length respectively of the box $\square \in\mu$
in the Young diagram associated with $\mu$.
\end{Lemma}

Now apply Lemma \ref{lem-eval-schur} to \eqref{eq-Z-Schurexp-1},
and we obtain the following:
\begin{equation*}
Z_{(r)}(s;v_1,\cdots,v_r;\bm p)=
\sum_{ \eta\in \cP}
\frac{s^{|\eta|} \cdot (|\eta|!)^{r-1}}{(\dim V^\eta)^{r-1}}
 \cdot s_\eta (\bm t)
\prod_{i=1}^{r} \prod_{\square \in \eta} \frac{v_i+c(\square)}{h(\square)}.
\end{equation*}
Recall that $\dim V^\eta$ is given by \eqref{eq-dim-hook},
and thus:
\be
\label{eq-Schurexp-1}
Z_{(r)}(s;v_1,\cdots,v_r;\bm p)=
\sum_{ \eta\in \cP} s^{|\eta|}
 \cdot
\Big( \prod_{\square \in \eta} \frac{ \prod_{i=1}^{r} \big(v_i+c(\square)\big)}{h(\square)}\Big)
\cdot s_\eta (\bm t).
\ee
One can further simplify this formula by using the following well-known evaluation formula
for Schur functions (see e.g. \cite[\S I.3. Example 5]{mac}):
\be
s_\eta (\delta_{k,1}) = \frac{1}{\prod_{\square \in \eta} h(\square)},
\ee
where $s_\eta (\delta_{k,1})$ is the evaluation of the Schur polynomial $s_\eta(\bm t)$
at $t_k = \delta_{k,1}$ for every $k\geq 1$.
In this way we can conclude that:
\begin{Proposition}
Take $t_n = \frac{p_n}{n}$,
then there is a Schur function expansion
for the generalized dessin partition function:
\be
\label{eq-Schurexp-Zr}
Z_{(r)}(s;v_1,\cdots,v_r;\bm p)=
\sum_{ \eta\in \cP} s^{|\eta|}\cdot
\Big( \prod_{\square \in \eta} \prod_{i=1}^{r} \big(v_i+c(\square)\big)\Big)
\cdot s_\eta(\delta_{k,1}) s_\eta (\bm t).
\ee
\end{Proposition}

\begin{Remark}
In the special case $r=2$,
the above result recovers the Schur function expansion
for the dessin partition function
(cf. \cite[(5)]{zhou2}):
\be
Z_{\text{dessin}} = \sum_{\mu\in \cP} s_\mu \cdot
\prod_{\square\in \mu} \frac{s\cdot \big(u+c(\square)\big)\big(v+c(\square)\big)}{h(\square)},
\ee
where $u=v_1$ and $v=v_2$.
\end{Remark}

\subsection{Cut-and-join representation for the generalized dessin partition function}

In this subsection we derive a cut-and-join type representation
for the generalized dessin partition functions.

First we need to recall the actions of $W$-type operators on Schur functions
derived by X. Liu and the second-named author in \cite{ly}.
Let $k\in \bZ_+$, and denote:
\be
\label{eq-def-Pmk-1}
P^{(k)} (z) =
:\big(\pd_z +J(z)\big)^{k-1} J(z):,
\ee
where $:\quad:$ is the normal ordering of bosonic operators (see e.g. \cite[\S 5]{djm} for the definition of the normal ordering)
and $J(z)$ is the following operator:
\be
\label{eq-def-Pmk-2}
J(z) = \sum_{n\in \bZ_+}  n t_n \cdot z^{n-1}
+\sum_{n\in \bZ_+} \frac{\pd}{\pd t_n} z^{-n-1}.
\ee
Define the operators $\{P_m^{(k)}\}_{m\in \bZ}$ by
\be
\label{eq-def-Pmk-3}
P^{(k)}(z) = \sum_{m\in \bZ} P_m^{(k)}\cdot z^{-m-k},
\ee
then the actions of these operators on Schur functions are given by:
\begin{Theorem}
[\cite{ly}]
\label{thm-LY-action}
Let $\mu = (\mu_1,\cdots,\mu_l)$ be a partition of integer.
Then for every $k\in \bZ_+$ and $m\in \bZ$,
one has:
\be
\label{eq-thm-LY-action}
\begin{split}
P_m^{(k)} s_\mu (\bm t)
=& \sum_{i=1}^{l} k[\mu_i-m-i]_{k-1} \cdot s_{\mu-m\epsilon_i} (\bm t)
+ \delta_{m,0} [-l]_k \cdot s_\mu(\bm t)  \\
& + \delta_{m<0}\cdot \sum_{n=1}^{-m} (-1)^{m-n}
k[n-l-1]_{k-1} \cdot s_{(\mu,n,1^{-m-n})}(\bm t) ,
\end{split}
\ee
where  $\mu-m\epsilon_i = (\mu_1,\cdots,\mu_{i-1}, \mu_i-m, \mu_{i+1},\cdots,\mu_l)$
(in the case where this sequence is no longer a partition,
see \S 2 in \cite{ly} for more details),
and
$[a]_k =
\prod_{j=0}^{k-1} (a-j)$
for an integer $a$.
Here we use the convention $[a]_0 =1$.
\end{Theorem}

Our main result in this subsection is the following:
\begin{Theorem}
\label{thm-cutandjoin-Z}
Let $a_1(\bm v),\cdots,a_{r+1}(\bm v)$ be the polynomials in $\bm v=(v_1,\cdots,v_r)$
determined by:
\be
\label{eq-def-akv}
\sum_{k=1}^{r+1} ka_k(\bm v) \cdot \prod_{j=0}^{k-2}(x-j)
= \prod_{i=1}^r (x+ v_i).
\ee
Then there is a cut-and-join type representation
for the generalized dessin partition function:
\be
\label{eq-cutandjoinrep-Z}
Z_{(r)}(s;v_1,\cdots,v_r;\bm p)
= e^{s\cdot \sum_{k=1}^{r+1} a_k(\bm v) P_{-1}^{(k)}} (1),
\ee
where $P_{-1}^{(k)}$ are the W-type operators defined by \eqref{eq-def-Pmk-1}-\eqref{eq-def-Pmk-3}.
\end{Theorem}
\begin{proof}
Take $m=-1$ in Theorem \ref{thm-LY-action},
and we obtain:
\be
\label{eq-P-1-k}
P_{-1}^{(k)} s_\mu (\bm t)
= \sum_{i=1}^{l(\mu)} k[\mu_i +1 -i]_{k-1} \cdot s_{\mu+\epsilon_i}(\bm t)
 + k[-l(\mu)]_{k-1} \cdot s_{(\mu,1)} (\bm t).
\ee
The right-hand side is actually a summation over
all possible ways of adding a box on the right or at the bottom of the
Young diagram associated with $\mu$
such that the resulting graph still represents a partition.
We denote by $\square$ this new box,
and denote by $\mu+\square$ the resulting partition.
If the new box $\square$ is in the $i$-th row where $1\leq i\leq l(\mu)$,
then its content is $c(\square) = \mu_i +1 -i$;
and if $\square$ is in the $(l(\mu)+1)$-th row,
its content is $c(\square) = -l(\mu)$.
Thus \eqref{eq-P-1-k} becomes:
\begin{equation*}
\begin{split}
P_{-1}^{(k)} s_\mu (\bm t)
= \sum_{ \mu + \square \in \cP}
k[c(\square)]_{k-1} \cdot s_{\mu+\square}(\bm t)
= \sum_{ \mu + \square \in \cP}
k \cdot \prod_{j=0}^{k-2} \big( c(\square)-j \big) \cdot s_{\mu+\square}(\bm t).
\end{split}
\end{equation*}
Then we have:
\begin{equation*}
\begin{split}
\Big( \sum_{k=1}^{r+1} a_k(\bm v) P_{-1}^{(k)} \Big) s_\mu
=& \sum_{k=1}^{r+1} a_k(\bm v) \sum_{\mu+\square \in \cP}
k\prod_{j=0}^{k-2} \big( c(\square) -j \big) s_{\mu+\square} (\bm t) \\
=& \sum_{\mu+\square \in \cP}
\Big( \prod_{i=1}^r \big( c(\square) + v_i \big) \Big)
\cdot s_{\mu+\square} (\bm t).
\end{split}
\end{equation*}
Notice that $s_{(\emptyset)} (\bm t) = 1$,
and thus:
\begin{equation*}
\begin{split}
e^{s\cdot \sum_{k=1}^{r+1} a_k(\bm v) P_{-1}^{(k)}} (1)
= &\sum_{n=0}^\infty \frac{s^n}{n!}\cdot
\Big( \sum_{k=1}^{r+1} a_k(\bm v) P_{-1}^{(k)} \Big)^n \big( s_{(\emptyset)} (\bm t) \big) \\
=& \sum_{\mu \in \cP} \frac{s^{|\mu|}}{|\mu|!} \cdot
K_\mu \cdot \prod_{\square \in \mu}
\prod_{i=1}^r \big( c(\square)+v_i\big)s_\mu(\bm t),
\end{split}
\end{equation*}
where $K_\mu$ is the number of all possible ways to add $n$ boxes successively to $(\emptyset)$ to obtain
the Young diagram of $\mu$,
i.e.,
the number of standard Young tableaux of shape $\mu$.
The number of standard Young tableaux of shape $\mu$ is known to be equal to the dimension
of the irreducible representation $V^\mu$ (see e.g. \cite[\S 4.3]{fh}):
\be
\label{eq-numberofYT}
K_\mu = \dim V^\mu =  \frac{|\mu|!}{\prod_{\square \in \mu} h(\square)},
\ee
And thus:
\begin{equation*}
\begin{split}
e^{s\cdot \sum_{k=1}^{r+1} a_k(\bm v) P_{-1}^{(k)}} (1)
=& \sum_{\mu \in \cP} s^{|\mu|} \cdot
\prod_{\square \in \mu}
\frac{ \prod_{i=1}^r \big( c(\square)+v_i\big) }{h(\square)}
s_\mu(\bm t).
\end{split}
\end{equation*}
Then \eqref{eq-cutandjoinrep-Z} follows from this equality and \eqref{eq-Schurexp-1}.
\end{proof}

It is well-known that the operators $P_m^{(k)}$
belong to the algebra $W_{1+\infty}$ which describes the symmetry of the KP hierarchy,
see e.g. \cite{fkn, ly}.
Therefore
the partition function
\begin{equation*}
Z_{(r)}(s;v_1,\cdots,v_r;\bm p)
= e^{s\cdot \sum_{k=1}^{r+1} a_k(\bm v) P_{-1}^{(k)}} (1)
\end{equation*}
is a tau-function of the KP hierarchy
with time variables $\bm t = (t_1,t_2,\cdots)$
where $t_n = \frac{p_n}{n}$.
As a straightforward consequence of the above theorem,
we have the following:
\begin{Corollary}
\label{cor-cutjoineq}
The following cut-and-join type equation holds:
\be
\frac{\pd}{\pd s} Z_{(r)}(s;v_1,\cdots,v_r;\bm p) =
\sum_{k=1}^{r+1} a_k(\bm v) \cdot P_{-1}^{(k)} Z_{(r)}(s;v_1,\cdots,v_r;\bm p).
\ee
\end{Corollary}

\begin{Example}
We give some examples of the above cut-and-join representation
for small $r$.
In the simplest case $r=1$,
Theorem \ref{thm-cutandjoin-Z} gives:
\begin{equation*}
Z_{(1)} (s;v_1;\bm p)
= \exp \Big[s\cdot \Big( v_1 P_{-1}^{(1)}
+ \half P_{-1}^{(2)}
\Big)\Big]\cdot 1,
\end{equation*}
where $t_n = \frac{p_n}{n}$ and
$P_{-1}^{(1)} = t_1$,
$P_{-1}^{(2)} = \sum_{n\geq 1} 2(n+1) t_{n+1} \frac{\pd}{\pd t_n}$.
For $r=2$,
the above theorem recovers the cut-and-join representation for the
dessin partition function in \cite{kz}:
\be
Z_{\text{dessin}} (s;v_1,v_2;\bm p)
= \exp \Big[s\cdot \Big( v_1v_2 P_{-1}^{(1)}
+ \frac{v_1+v_2+1}{2} P_{-1}^{(2)}
+ \frac{1}{3} P_{-1}^{(3)}
\Big)\Big]\cdot 1,
\ee
where $P_{-1}^{(3)} =
3 \sum_{k,l\geq 1} \Big(klt_kt_l \frac{\pd}{\pd t_{k+l-1}}
+ (k+l+1)t_{k+l+1}\frac{\pd}{\pd t_k}\frac{\pd}{\pd t_l} \Big)
-3\sum_{k\geq 1} (k+1)t_{k+1}\frac{\pd}{\pd t_k}$.
And for $r=3$,
the cut-and-join representation matches with the result given in \cite[\S 8.4]{ammn}:
\begin{equation*}
Z_{(3)} (s;\bm v;\bm p)
= \exp\Big[s \Big( e_3(\bm v) P_{-1}^{(1)}
+ \frac{e_2(\bm v)+e_1(\bm v)+1}{2} P_{-1}^{(2)}
+ \frac{e_1(\bm v)+3}{3} P_{-1}^{(3)}
+ \frac{1}{4} P_{-1}^{(4)}
\Big)\Big]\cdot 1,
\end{equation*}
where we use $e_k(\bm v)$ to denote the elementary symmetric function
of degree $k$ in $\bm v =(v_1,\cdots,v_r)$.
For $r=4$,
the result is:
\begin{equation*}
\begin{split}
&Z_{(4)} (s;\bm v;\bm p)
= \exp\Big[ s \Big(
e_4(\bm v) P_{-1}^{(1)} + \half \big(e_0(\bm v) + e_1(\bm v)+e_2(\bm v)+e_3(\bm v) \big) P_{-1}^{(2)} \\
&\quad
 + \frac{1}{3}\big( 7e_0(\bm v) + 3e_1(\bm v) +e_2(\bm v) \big)P_{-1}^{(3)}
+\frac{1}{4}\big(6e_0(\bm v) +e_1(\bm v) \big)P_{-1}^{(4)}
+ \frac{e_0(\bm v)}{5} P_{-1}^{(5)}
\Big) \Big] \cdot 1,
\end{split}
\end{equation*}
\end{Example}

At the end of this subsection,
we give an explicit formula for the coefficients $a_k(\bm v)$ in Theorem \ref{thm-cutandjoin-Z}
and Corollary \ref{cor-cutjoineq}
(see Appendix \ref{app-pflem5.2} for a proof):
\begin{Lemma}
\label{lem-ak-explicit}
Let $r\geq 1$.
Then the functions $\{a_k(\bm v)\}_{k=1}^r$ are given by:
\be
a_k (\bm v) =
\frac{(-1)^{k-1}}{k!}e_r(\bm v)
+\frac{1}{k}
\sum_{j=0}^r \Big( \sum_{i=1}^{k-1} \frac{(-1)^{i+k-1} \cdot i^{r-j}}{i!\cdot (k-1-i)!} \Big)  e_j(\bm v),
\ee
where $e_j(\bm v)$ is the elementary symmetric polynomial of degree $j$ in $(v_1,\cdots,v_r)$,
and $e_0 (\bm v)= 1$.
\end{Lemma}

\subsection{$\bN$-coefficient binomial polynomiality for the generalized dessin counting}

As a corollary of Theorem \ref{thm-binompol-H},
we have the following:
\begin{Theorem}
\label{thm-pol-gdessin}
Fix some positive integers $l>0$, $r>1$ and $k_1,k_2,\cdots,k_r>0$.
Then for a partition of integer $\mu = (\mu_1, \mu_2, \cdots, \mu_l)$ (with $|\mu|>k_i$ for every $i$),
the number
\be
z_\mu \cdot N_{|\mu|-k_1,\cdots,|\mu|-k_{r-1},k_r}^\circ (\mu)
\ee
is a linear combination of
\be
\label{eq-binomterm-N}
\binom{\mu_1}{d_1} \binom{\mu_2}{d_2} \cdots \binom{\mu_l}{d_l}
\ee
with non-negative integer coefficients,
where $d_j\geq 1$ for every $1\leq j\leq l$ and
\be
\sum_{j=1}^l d_j \leq 2\sum_{i=1}^{r-1} k_i.
\ee
\end{Theorem}
\begin{proof}
For a partition $\lambda$,
we say $\lambda$ is minimal if it does not contain $1$,
i.e., if $m_1(\lambda) = 0$.
Notice that for arbitrary partitions $\lambda^1,\cdots,\lambda^r$
and non-negative integers $a_1,\cdots,a_r$,
one always has:
\begin{equation*}
H_{\lambda^1,\cdots,\lambda^r;k}^\circ (\mu)
= H_{(\lambda^1,1^{a_1}),\cdots,(\lambda^r,1^{a_r});k}^\circ (\mu)
\end{equation*}
for every $\mu$ satisfying $|\mu|\geq |\lambda^i|+a_i$,
where $H_{\lambda^1,\cdots,\lambda^r;k}^\circ (\mu)$ is
the Hurwitz numbers discussed in \S \ref{sec-Npol-H-pf}.
Then by the definition of $H_{\lambda^1,\cdots,\lambda^r;k}^\circ (\mu)$,
we have:
\be
\label{eq-pf-dessinpol}
N_{|\mu|-k_1,\cdots,|\mu|-k_{r-1},k_r}^\circ (\mu)
= \sum_{\substack{\lambda^1,\cdots,\lambda^{r-1}\in \cP^{\text{min}}\\ l(\tilde\lambda^i) = |\mu|-k_i}}
H_{\lambda^1,\cdots,\lambda^{r-1};k_r}^\circ (\mu),
\ee
where $\cP^{\text{min}}$ is the set of all minimal partitions,
and $\tilde\lambda^i$ is the partition $(\lambda^i,1^{|\mu|-|\lambda^i|})$.
Here the condition $l(\tilde\lambda^i) = |\mu|-k_i$ is equivalent to:
\begin{equation*}
l(\lambda) + |\mu| - |\lambda| = |\mu|-k_i,
\end{equation*}
i.e., $|\lambda| - l(\lambda) = k_i$.
Notice that for a fixed $k_i$,
there is only a finite number of minimal partitions $\lambda^i$ satisfying this condition,
and thus the summation in the right-hand side of \eqref{eq-pf-dessinpol} is a finite sum.
Then the conclusion follows from \eqref{eq-pf-dessinpol} and Theorem \ref{thm-binompol-H}.
\end{proof}

\begin{Remark}
\label{rmk-pol-gdessin-coeff}
One can also write down a proof straightforwardly using the same method in \S \ref{sec-Npol-H-pf}.
In fact,
denote $n = |\mu|$,
and let $\gamma_0 \in S_n$ and $\Gamma_s \subset [n]$
be defined by \eqref{eq-def-gamma0} and \eqref{eq-def-Gammas} respectively.
For every $1\leq s \leq l$,
fix $d_s$ distinct numbers
$b_{s,1},b_{s,2},\cdots, b_{s,d_s} \in \Gamma_s$,
and then the coefficient of \eqref{eq-binomterm-N} is the size of the following set:
\be
\label{eq-gdessin-coeff}
\Bigg\{ (\alpha_1,\cdots,\alpha_{r-1},\beta) \in  (S_n)^r \Bigg|
\begin{array}{l}
l(\alpha_i) =n-k_i,
\text{ } l(\beta) = k_r,
\text{ } \alpha_1\cdots\alpha_{r-1}\beta\gamma_0 = e, \\
\text{$\langle \alpha_1,\cdots,\alpha_{r-1},\beta,\gamma_0 \rangle$ acts transitively on $[n]$,} \\
\text{and $F\cap \Gamma_s = \bm b_s$ for every $s$}
\end{array}
\Bigg\}
\ee
which is actually independent of $\mu$ and the choice of the numbers $b_{s,j}$,
where $F\subset [n]$ is the subset of elements which are not fixed by at least one of
$\alpha_1,\cdots,\alpha_{r-1}$,
and $l(\alpha_i)$ is number of cycles in $\alpha_i \in S_n$.
\end{Remark}

\begin{Corollary}
Fix some positive integers $l>0$, $r>1$ and $k_1,k_2,\cdots,k_r>0$.
Then for a partition of integer $\mu = (\mu_1, \mu_2, \cdots, \mu_l)$ (with $|\mu|>k_i$ for every $i$),
the number
\be
z_\mu \cdot N_{|\mu|-k_1,\cdots,|\mu|-k_{r-1},k_r}^\circ (\mu)
\ee
is a polynomial with rational coefficients in $\mu_1,\cdots,\mu_l$
of degree less than or equals to $2\sum_{i=1}^{r-1} k_i$.
\end{Corollary}

\subsection{Examples}

Now we give some examples of the generalized dessin counting.

\begin{Example}
Consider the simplest case $l(\mu) = 1$,
and denote $\mu = (n)$.
For $r=2$,
the coefficient of $\binom{n}{d}$ in $n\cdot N_{n-k_1,k_2}^\circ \big( (n) \big)$
is the size of the set:
\begin{equation*}
\Bigg\{ (\alpha,\beta) \in  (S_n)^2 \Bigg|
\begin{array}{l}
l(\alpha) =n-k_1,
\text{ } l(\beta) = k_2,
\text{ } \alpha\beta (1,2,\cdots, n) = e, \\
\text{$\langle \alpha,\beta, (1,2,\cdots, n) \rangle$ acts transitively on $[n]$,} \\
\text{and $F = \{1,2,\cdots,d\}$ for every $s$}
\end{array}
\Bigg\},
\end{equation*}
where $F\subset [n]$ is the subset of elements which are not fixed by $\alpha$.

First for $k_1=k_2=1$,
one has
\begin{equation*}
nN_{n-1,1}^\circ \big( (n) \big) =0
\end{equation*}
since in this case $l(\alpha)=n-1$ implies $\alpha$ is a transposition
and $l(\beta) = 1$ implies $\beta$ is a $n$-cycle,
and thus $\alpha\beta (1,2,\cdots, n)$ cannot be $e$ since it is an odd permutation.

And for $k_1=1$ and $k_2=2$,
the coefficient of $\binom{n}{d}$ is:
\begin{equation*}
\Big|
\big\{ \alpha \in S_n \big|
l(\alpha) =n-1,
\text{ } l\big(\alpha^{-1}\cdot (n,n-1,\cdots,1)\big) = 2,
\text{ } F = \{1,2,\cdots,d\}
\big\}
\Big|.
\end{equation*}
Here $l(\alpha) =n-1$ implies that $\alpha$ is a transposition,
and thus $d=2$ and $\alpha = (12)$.
One easily checks that $l\big((12)\cdot (n,n-1,\cdots,1)\big) = 2$ indeed holds,
and thus
\begin{equation*}
nN_{n-1,2}^\circ \big( (n) \big) = \binom{n}{2} = \frac{n(n-1)}{2}.
\end{equation*}

For $k_1=2$ and $k_2=1$,
the coefficient of $\binom{n}{d}$ is:
\begin{equation*}
\Big|
\big\{ \alpha \in S_n \big|
l(\alpha) =n-2,
\text{ } l\big(\alpha^{-1}\cdot (n,n-1,\cdots,1)\big) = 1,
\text{ } F = \{1,2,\cdots,d\}
\big\}
\Big|.
\end{equation*}
Here the cycle type of $\alpha$ must be $(2,2,1,\cdots,1)$ or $(3,1,\cdots,1)$.
If $\alpha$ is of type $(2,2,1,\cdots,1)$,
then $d=4$ and $\alpha=(12)(34)$, $(13)(24)$, or $(14)(23)$.
One checks that $\alpha=(13)(24)$ is the only one satisfying
$l\big(\alpha^{-1}\cdot (n,n-1,\cdots,1)\big) = 1$.
Similarly if $\alpha$ is of type $(3,1,\cdots,1)$,
then $d=3$ and $\alpha=(123)$ or $(132)$,
and $l\big(\alpha^{-1}\cdot (n,n-1,\cdots,1)\big) = 1$ tells that $\alpha$ must be $(132)$.
We conclude that:
\begin{equation*}
nN_{n-2,1}^\circ \big( (n) \big) = \binom{n}{3} + \binom{n}{4}
= \frac{1}{24}(n-2)(n-1)n(n+1).
\end{equation*}
\end{Example}

\begin{Example}
The following are some examples for $l=1$ and $r=3$.
The simplest case $k_1=k_2=k_3 = 1$ has been computed in Example \ref{eg-3.3pol},
and the result is:
\begin{equation*}
nN_{n-1,n-1,1}^\circ \big((n)\big) =
2\binom{n}{4} + 3\binom{n}{3} + \binom{n}{2} = \frac{1}{12} n^2(n-1)(n+1).
\end{equation*}

By the Riemann-Hurwitz formula \eqref{eq-RiemHur} one has:
\begin{equation*}
nN_{n-1,n-1,2}^\circ \big((n)\big) =
nN_{n-1,n-2,1}^\circ \big((n)\big) =
nN_{n-2,n-1,1}^\circ \big((n)\big) =0.
\end{equation*}

Now consider $nN_{n-1,n-2,2}^\circ \big((n)\big)$.
In this case the degree of the polynomial is not greater than
$2(1+2) = 6$.
The coefficient of $\binom{n}{6}$ is the number of the pairs $(\alpha_1,\alpha_2)\in (S_n)^2$
such that
\be
\label{eq-eg-122-cond}
l(\alpha_1)=n-1,\qquad
l(\alpha_2)=n-2,\qquad
l\big(\alpha_2^{-1}\alpha_1^{-1}\cdot (n,n-1,\cdots,1) \big) =2,
\ee
and the set of elements not fixed by $\alpha_1$ or $\alpha_2$ is $\{1,2,\cdots,6\}$.
In this case $\alpha_1$ is a transposition $(ab)$,
and $\alpha_2$ must be of the form $(cd) (ef)$ such that
$\{a,b,c,d,e,f\} = \{1,2,\cdots,6\}$.
A simple calculation tells that \eqref{eq-eg-122-cond} implies:
\begin{equation*}
\begin{split}
(ab)(cd)(ef) =& (12)(35)(46),\quad
(23)(46)(51),\quad
(34)(51)(62),\quad
(45)(62)(13),\\
& (56)(13)(24),\quad
(61)(24)(35),\quad
(14)(25)(36),\quad
(14)(26)(35),\\
& (25)(13)(46),\quad
(36)(24)(51),
\end{split}
\end{equation*}
and there are $\binom{3}{1}\cdot 10 = 30$ choices of such a pair $(\alpha_1,\alpha_2)\in (S_n)^2$.
Therefore the leading term of $nN_{n-1,n-2,2}^\circ \big((n)\big)$
is $30\binom{n}{6}$.
Similarly one can compute other terms and finally obtain (here we omit the details):
\begin{equation*}
nN_{n-1,n-2,2}^\circ \big((n)\big)
= 30\binom{n}{6} + 60 \binom{n}{5} +36 \binom{n}{4} +6\binom{n}{3}
= \frac{1}{24} (n-2)(n-1)^2 n^2 (n+1) ,
\end{equation*}
\end{Example}

\begin{Example}
For $l=2$ and $r=2$ we have
(here we omit the details of computations):
\begin{equation*}
\begin{split}
 z_{(n_1,n_2)} \cdot N_{n_1+n_2-1,1}^\circ \big( (n_1,n_2)\big)
=& \binom{n_1}{1} \binom{n_2}{1} = n_1n_2, \\
 z_{(n_1,n_2)} \cdot N_{n_1+n_2-2,2}^\circ \big( (n_1,n_2)\big)
=& 3\binom{n_1}{3}\binom{n_2}{1} + 3\binom{n_1}{1}\binom{n_2}{3} +2 \binom{n_1}{2}\binom{n_2}{2} \\
 &   +2\binom{n_1}{2}\binom{n_2}{1}+2\binom{n_1}{1}\binom{n_2}{2} \\
 =& \half n_1n_2(n_1^2+n_2^2+n_1n_2 -2n_1-2n_2+1), \\
z_{(n_1,n_2)} \cdot N_{n_1+n_2-1,2}^\circ \big( (n_1,n_2)\big)
=& z_{(n_1,n_2)} \cdot N_{n_1+n_2-2,1}^\circ \big( (n_1,n_2)\big) =0.
\end{split}
\end{equation*}
\end{Example}

\section{Application: Polynomiality of One-Part Double Hurwitz Numbers}
\label{sec-onepart}

In this section
we apply the main theorem to obtain the polynomiality
of the ordinary one-part double Hurwitz numbers
and one-part double Hurwitz numbers with completed cycles.
The conclusions match with some results in literatures \cite{gjv2, ssz}.

\subsection{Polynomiality for ordinary one-part double Hurwitz numbers}

In this subsection we recall the definition of one-part double Hurwitz numbers
and prove the polynomiality.
The polynomiality of double Hurwitz numbers has been proved by Goulden-Jackson-Vakil in \cite{gjv2}.

The ordinary double Hurwitz number labeled by two partitions $\mu^\pm$
is the weighted counting of the branched coverings of $\bC\bP^1$ which has ramification type $\mu^\pm$
over two given points,
and simple ramifications over other branch points.
It is known that the generating series of double Hurwitz numbers
is a tau-function of the $2$d Toda lattice hierarchy \cite{ok,op2}.
See \cite{wy} for an algorithm computing the connected double Hurwitz numbers
using the boson-fermion correspondence.

Fix a positive integer $d$, and denote $\mu^0 = (2,1^{d-2})$.
Then for two partitions $\mu^\pm$ of $d$,
a double Hurwitz number labeled by $\mu^\pm$ is of the form:
\begin{equation*}
H_g^\circ (\mu^0,\cdots,\mu^0,\mu^+,\mu^-),
\end{equation*}
where the number of $\mu^0$ is
\begin{equation*}
2g-2 + l(\mu^+) +l(\mu^-)
\end{equation*}
by the Riemann-Hurwitz formula \eqref{eq-RiemHur}.
If one takes $\mu^+ = (d)$,
then the resulting Hurwitz number $H_{g}^\circ (\mu^0,\cdots,\mu^0,(d),\mu)$ is called a one-part double Hurwitz number.
It is clear that in the one-part case
the connected Hurwitz numbers coincide with the disconnected ones.

Notice that $l(\mu^0) = d-1$.
Now we denote $\mu^-$ by $\mu$,
and then the one-part double Hurwitz number is
a special case of the generalized dessin counting:
\begin{equation*}
H_{g}^\circ (\mu^0,\cdots,\mu^0,(d),\mu) =
N_{|\mu|-1,\cdots,|\mu|-1,1}^\circ \big( (\mu) \big).
\end{equation*}
Then by Theorem \ref{thm-pol-gdessin} one has the following $\bN$-coefficient binomial polynomiality:

\begin{Proposition}
Let $\mu = (\mu_1,\cdots,\mu_l)$ be a partition of integer of a fixed length $l$,
and denote $d=|\mu|$.
Then the one-part double Hurwitz number $z_\mu \cdot H_{g}^\circ (\mu^0,\cdots,\mu^0,(d),\mu)$
(where $\mu^0 = (2,1,\cdots,1)$ and the number of $\mu^0$ is fixed)
is a linear combination of
\begin{equation*}
\binom{\mu_1}{d_1} \binom{\mu_2}{d_2} \cdots \binom{\mu_l}{d_l}
\end{equation*}
with non-negative integer coefficients,
where $d_j\geq 1$ for every $1\leq j\leq l$ and
\be
\sum_{j=1}^l d_j \leq 2\big(2g-1 +l(\mu)\big).
\ee
And thus $z_\mu \cdot H_{g}^\circ (\mu^0,\cdots,\mu^0,(d),\mu)$ is a polynomial
of rational coefficients in $\mu_1,\cdots,\mu_l$
of degree not greater than $4g-2+ 2l(\mu)$.
\end{Proposition}

\begin{Remark}
In \cite{gjv2},
Goulden-Jackson-Vakil proved the piecewise polynomiality of double Hurwitz numbers.
The case of one-part double Hurwitz numbers is more interesting in some sense because
it has some conjectural connections to the intersection numbers
on some compactification of Picard varieties,
see \cite[Conjecture 3.5]{gjv2}.
Taking $\alpha = (d)$ in \cite[Theorem 2.1]{gjv2} tells that
$|\Aut(\mu)|\cdot H_{g}^\circ (\mu^0,\cdots,\mu^0,(d),\mu)$ is
a polynomial in $\mu_1,\cdots,\mu_l$ of degree up to $4g-2+l$.
This matches with the above proposition since $z_\mu = |\Aut(\mu)| \cdot \mu_1\cdots \mu_l$,
and $\mu_1\cdots\mu_l \big| \binom{\mu_1}{d_1} \cdots \binom{\mu_l}{d_l}$
for $d_i\geq 1$.
(Notice here the notations for Hurwitz numbers in this work and \cite{gjv2}
differ by the a factor which is the product of the numbers of automorphisms of the partitions).
\end{Remark}

\subsection{One-part double Hurwitz numbers with completed cycles}

The case of one-part double Hurwitz numbers with completed cycles is similar.
The polynomiality of double Hurwitz numbers with completed cycles
has been proved by Shadrin-Spitz-Zvonkine in \cite{ssz}.
In this subsection we understand the polynomiality of one-part double Hurwitz numbers
with completed cycles
as a consequence of Theorem \ref{thm-binompol-H}.

First we recall the definition of the completed cycles,
see \cite[\S 0.4]{op} and \cite[\S 2]{ssz}.
For a positive integer $k$,
denote by $p_k^{\text{sh}}$ the following shifted symmetric function:
\begin{equation*}
p_k^{\text{sh}} (x_1,x_2,x_3,\cdots) =
\sum_{i=1}^\infty \Big( (x_i - i+\half)^k - (-i+\half)^k \Big).
\end{equation*}
And for a partition $\mu = (\mu_1, \mu_2,\cdots,\mu_l)$,
denote $p_\mu^{\text{sh}} = p_{\mu_1} p_{\mu_2}\cdots p_{\mu_l}$.
Then $\{p_\mu^{\text{sh}}\}_{\mu\in \cP}$ forms a basis of the space of all shifted symmetric functions $\Lambda^*$.
Let $\mu,\lambda$ be two partitions of $d$,
and let $C_\mu\subset S_d$ be the conjugate class of permutations whose cycle type is $\mu$.
Denote:
\begin{equation*}
f_\mu (\lambda) = |C_\mu|\cdot \frac{\chi^\lambda (C_\mu)}{ \dim V^\lambda}
\end{equation*}
where
$V^\lambda$ denotes the irreducible representation of $S_d$ indexed by the partition $\lambda$,
and $\chi^\lambda$ is the irreducible character associated with $V^\lambda$.
And this definition can be extended to the case $|\mu|< |\lambda|$ in the following way:
\begin{equation*}
f_\mu (\lambda) = \binom{|\lambda|}{|\mu|} |C_\mu| \cdot
\frac{\chi^\lambda (C_{\tilde\mu})}{ \dim V^\lambda}
= \binom{m_1(\tilde\mu)}{m_1(\mu)} |C_{\tilde\mu}| \cdot
\frac{\chi^\lambda (C_{\tilde\mu})}{ \dim V^\lambda},
\end{equation*}
where $\tilde\mu = (\mu,1,\cdots,1)$ such that $|\tilde\mu | = |\lambda|$,
and $m_1(\mu)$ is the number of $1$ appearing in $\mu$.
And $f_\mu (\lambda)$ vanishes for $|\mu|> |\lambda|$.
It is proved by Kerov and Olshanski \cite{ko} that
$\{f_\mu\}_{\mu\in \cP}$ forms another basis for the space $\Lambda^*$.

Now let $\bC S_d$ be the group algebra of the symmetric group $S_d$,
and let $Z\bC S_d$ be its center,
i.e., the class algebra.
Then $Z\bC S_d$ has a basis $\{ \widehat C_\lambda \}_{|\lambda| = d}$
where
\begin{equation*}
\widehat C_\lambda = \sum_{\sigma \in C_\lambda} \sigma,
\end{equation*}
and $C_\lambda\subset S_d$ is the conjugate class of permutations whose cycle type is $\lambda$.
Then there exists a linear isomorphism $\phi : \bigoplus_{d=0}^\infty Z\bC S_d \to \Lambda^*$ such that
$\phi(\widehat C_\lambda) = f_\lambda$ for every $\lambda \in \cP$.
For a partition $\mu = (\mu_1,\cdots,\mu_l)$,
the completed $\mu$-conjugate class $\overline{C}_\mu$ is defined by:
\begin{equation*}
\overline C_\mu = \frac{\phi^{-1} (p_\mu)}{\prod_{i=1}^l \mu_i !}.
\end{equation*}

The disconnected double Hurwitz numbers with completed $(r+1)$-cycles
labeled by two partitions $\lambda, \mu$ of $d$ are defined as follows:
\begin{equation*}
H_{g;\lambda,\mu}^{(r)} = \sum_{|\eta|=d}
\frac{\chi^\eta(C_\lambda)}{z_\lambda }
\cdot \Big( \frac{p_{r+1}(\eta)}{(r+1)!} \Big)^s
\cdot \frac{ \chi^\eta (C_\mu)}{ z_\mu},
\end{equation*}
where $z_\mu = d!/|C_\mu|$,
and $s$ is determined by
\begin{equation*}
2g-2 = rs-l(\lambda) - l(\mu).
\end{equation*}
This is actually defined similarly as the ordinary double Hurwitz numbers such that
the simple branches are all replaced by the completed $(r+1)$-conjugate class $\overline{C}_{(r+1)}$,
see the Burnside character formula \eqref{eq-Burnside}.
(Here our notation $H_{g;\mu,\nu}^{(r)}$ differs from the notation $h_{g,\mu,\nu}^{(r)}$ in \cite{ssz}
by the factor $|\Aut(\mu)|\cdot |\Aut(\nu)|$.)

It is known that the completed $(r+1)$-cycle $\overline{C}_{(r+1)}\in  \bigoplus_{d=0}^\infty Z\bC S_d$
is a linear combination of $\widehat C_\nu $
where $|\nu| \leq r+1$,
with non-negative rational coefficients
(see \cite[\S 0.4.4]{op}):
\begin{equation*}
\begin{split}
& 0! \cdot \overline{(1)} = (1),\\
& 1! \cdot \overline{(2)} = (2),\\
& 2! \cdot \overline{(3)} = (3) + (1,1) + \frac{1}{12}\cdot (1),\\
& 3! \cdot \overline{(4)} = (4) + 2\cdot (2,1) + \frac{5}{4}\cdot (2).
\end{split}
\end{equation*}
Thus for two partitions $\lambda,\mu$ with $|\lambda| = |\mu| = d \geq r+1$,
the Hurwitz number with completed cycles $H_{g;\lambda,\mu}^{(r)}$
is a linear combination of
\be
\label{eq-onepart-coeffbinom}
\prod_{i=1}^s \binom{m_1(\tilde \nu^i)}{m_1(\nu^i)}
\cdot  H_g^\bullet (\lambda,\mu,\tilde\nu^1,\cdots,\tilde\nu^s),
\qquad |\nu^i| \leq r+1,
\ee
with non-negative rational coefficients,
where $\tilde \nu^i = (\nu^i,1^{d-|\nu^i|})$.

Now consider the case of the one-part double Hurwitz number with completed cycles
$H_{g;\mu, (d)}^{(r)}$ where $\mu = (\mu_1,\cdots,\mu_l)$ is a partition of $d$
with fixed length $l$.
In the one-part case the connected Hurwitz numbers coincide with the disconnected ones.
We have:
\begin{Proposition}
Let $\mu = (\mu_1,\cdots,\mu_l)$ be a partition of fixed length $l$
with $|\mu| \geq r+1$,
and denote $d= |\mu|$.
Then the one-part double Hurwitz number with completed cycles $z_\mu \cdot H_{g;\mu, (d)}^{(r)}$
is an element in a polynomial in $\mu_1,\cdots,\mu_l$ with rational coefficients.
\end{Proposition}
\begin{proof}
In the one-part case,
the number $z_\mu\cdot H_{g;\mu,(d)}^{(r)}$
is a linear combination of
\begin{equation*}
\prod_{i=1}^s \binom{m_1(\tilde \nu^i)}{m_1(\nu^i)}
\cdot  z_\mu\cdot H_{g;\nu^1,\cdots,\nu^s; 1}^\circ \big( (\mu) \big),
\qquad |\nu^i| \leq r+1,
\end{equation*}
with rational coefficients.
Here for fixed $\nu^1,\cdots,\nu^s$,
Theorem \ref{thm-binompol-H} tells that
$H_{g;\nu^1,\cdots,\nu^s;1}^\circ \big((\mu)\big)$
is a linear combination of the polynomials
\begin{equation*}
\binom{\mu_1}{d_1} \binom{\mu_2}{d_2} \cdots \binom{\mu_l}{d_l},
\end{equation*}
where $d_j\geq 1$ and $\sum_{j=1}^l d_j \leq 2 \sum_{i=1}^s (|\nu^i| - l(\mu^i))$.
Moreover,
the binomial coefficient
\begin{equation*}
\binom{m_1(\tilde \nu^i)}{m_1(\nu^i)} =
\binom{d-|\nu^i| + m_1(\nu^i)}{m_1(\nu^i)}
\end{equation*}
in \eqref{eq-onepart-coeffbinom} is a polynomial in $d$
and thus is a polynomial in $\mu_1,\cdots,\mu_l$ for fixed $\nu^i$.
Therefore the conclusion holds.
\end{proof}

The same argument can be applied to obtain the polynomiality of $z_\mu\cdot H_{g;m;\mu}^{(r)}$,
where $H_{g;m;\mu}^{(r)}$ is the following $m$-part double Hurwitz numbers with completed cycles:
\begin{equation*}
H_{g;m;\mu}^{(r)} =  \sum_{l(\nu) = m} H_{g;\mu,\nu}^{(r)}.
\end{equation*}
In the $m$-part case (for $m\geq 2$) the connected and the disconnected Hurwitz numbers
do not coincide anymore,
but the polynomiality is still valid in both cases,
see Remark \ref{rmk-disconnpol}.

\vspace{.3in}

{\bf Acknowledgements}.
The authors would like to thank Professor Maciej Dołęga for telling us
the relation between our work and the $b$-deformed Hurwitz numbers,
and thank an anonymous referee for pointing out that
the polynomiality of the generalized dessin partition function
can be derived from the topological recursion.
The authors would also like to thank Professor Xiaobo Liu and Professor Jian Zhou for helpful discussions and encouragement.
The second-named author is supported by the NSFC grant (No. 12288201, 12401079),
the China Postdoctoral Science Foundation (No. 2023M743717),
and the China National Postdoctoral Program for Innovative Talents (No. BX20240407).

\begin{appendix}

\section{Proofs of Two Lemmas}

\subsection{Proof of Lemma \ref{lem-eval-schur}}
\label{app-pflem5.1}

Here we give a new proof of Lemma \ref{lem-eval-schur}.
Our strategy is to use the action of $W$-operators $P_m^{(k)}$ on Schur functions proposed by Liu and the second-named author in \cite{ly}.

First we take $m=\pm 1$ and $k=1,2$ in Theorem \ref{thm-LY-action},
and obtain:
\be
\label{eq-egcutandjoin-m=-1}
t_1\cdot s_\mu (\bm t)
= \sum_{ \mu + \square \in \cP}
 s_{\mu+\square}(\bm t), \qquad
L_{-1} s_\mu (\bm t)
= \sum_{ \mu + \square \in \cP}
 c(\square)\cdot s_{\mu+\square}(\bm t),
\ee
and
\be
\label{eq-egcutandjoin-m=1}
\frac{\pd}{\pd t_1} s_\mu (\bm t)
= \sum_{ \mu - \square \in \cP}
 s_{\mu-\square}(\bm t), \qquad
L_1 s_\mu (\bm t)
= \sum_{ \mu - \square \in \cP}
 c(\square)\cdot s_{\mu-\square}(\bm t),
\ee
where
\begin{equation*}
\begin{split}
L_{-1}  =
\sum_{n\geq 1} (n+1) t_{n+1} \frac{\pd}{\pd t_n},\qquad\qquad
L_1 = \sum_{n\geq 1} nt_n \frac{\pd}{\pd t_{n+1}}
\end{split}
\end{equation*}
are the Virasoro operators.
Here we denote by $\mu - \square$ a partition whose Young diagram is obtained from
that of $\mu$ by deleting a box $\square$ on the right of some row,
and set $s_{\mu-\square} = 0$ if $\mu-\square$ is no longer a partition.
The above formulas for $t_n s_\mu$ and $\frac{\pd}{\pd t_n} s_\lambda$ are well-known in literatures,
and formulas for $k=2$ can also be found in \cite{am,ly2}.
Now let $v$ be a formal variable,
and we prove the following:
\begin{Lemma}
We have:
\be
\label{eq-lem-L-1+t1}
e^{s(L_{-1}+v \cdot t_1)} (1) =
\sum_{\mu \in \cP} s^{|\mu|} \cdot s_\mu (t_k = \frac{v}{k}) \cdot s_\mu(\bm t).
\ee
\end{Lemma}
\begin{proof}
Denote by $G(\bm t)$ the right-hand side of \eqref{eq-lem-L-1+t1}.
By \eqref{eq-egcutandjoin-m=-1} we have:
\begin{equation*}
\begin{split}
(L_{-1}+vt_1) G(\bm t)
=& \sum_{\mu\in \cP} s^{|\mu|} \cdot s_\mu (t_k = \frac{v}{k})
\cdot \sum_{\mu+\square \in \cP} \big( c(\square) +v \big) s_{\mu+\square} (\bm t)\\
=&
\sum_{\lambda\in \cP} s^{|\lambda| -1}  \cdot
\sum_{\lambda - \square \in \cP} \big( c(\square) +v \big)
s_{\lambda-\square} (t_k = \frac{v}{k}) \cdot s_{\lambda } (\bm t),
\end{split}
\end{equation*}
where we denote $\lambda = \mu+\square$.
Then by \eqref{eq-egcutandjoin-m=1} we see that:
\begin{equation*}
(L_{-1}+vt_1) G(\bm t) =
 \sum_{\lambda\in \cP} s^{|\lambda| -1}  \cdot
\Big(v \frac{\pd}{\pd t_1} + \sum_{n\geq 1} p_n\cdot (n+1)\frac{\pd}{\pd p_{n+1}} \Big)
s_{\lambda} (\bm t)\big|_{p_k =v}
\cdot s_{\lambda } (\bm t)
\end{equation*}
where $p_k = kt_k$ for every $k$.
Recall that if we set $\deg (t_k) = \deg(p_k) =k$,
then $s_\lambda (\bm t)$ is a weighted homogeneous polynomial of degree $|\mu|$.
Thus by Euler's formula,
\begin{equation*}
\Big(v \frac{\pd}{\pd t_1} + \sum_{n\geq 1} p_n\cdot (n+1)\frac{\pd}{\pd p_{n+1}} \Big)
s_{\lambda} (\bm t)\big|_{p_k =v}
= |\lambda| \cdot s_\lambda (\bm t) \big|_{p_k =v}
= |\lambda| \cdot s_\lambda (t_k = \frac{v}{k}),
\end{equation*}
and then:
\begin{equation*}
\begin{split}
(L_{-1}+vt_1) G(\bm t) = \sum_{\lambda\in \cP} s^{|\lambda|-1} \cdot |\lambda| \cdot
s_\lambda (t_k = \frac{v}{k})\cdot s_\lambda(\bm t)
= \frac{\pd}{\pd s} G(t).
\end{split}
\end{equation*}
Notice that $G(\bm t)|_{s=0} = 1$, and then we conclude that
$G(\bm t) = e^{s(L_{-1}+v \cdot t_1)} (1)$.
\end{proof}

Now we can present a proof of Lemma \ref{lem-eval-schur}.
We expand the exponential in the left-hand side of \eqref{eq-lem-L-1+t1} directly and then apply \eqref{eq-egcutandjoin-m=-1},
and obtain:
\begin{equation*}
\begin{split}
e^{s(L_{-1}+v \cdot t_1)} (1)
= \sum_{n=0}^\infty \frac{s^n}{n!}  (L_{-1}+v \cdot t_1)^n s_{(\emptyset)} (\bm t)
= \sum_{n=0}^\infty \frac{s^n}{n!}
\sum_{|\mu| = n} K_\mu  \prod_{\square \in \mu} \big( c(\square)+v\big)s_\mu(\bm t),
\end{split}
\end{equation*}
where $K_\mu$ is the number of standard Young tableaux of shape $\mu$.
By \eqref{eq-numberofYT} we have:
\be
e^{s(L_{-1}+v \cdot t_1)} (1)
= \sum_{\mu \in \cP} s^{|\mu|}\prod_{\square \in \mu}\frac{ \big( c(\square)+v\big) }{h(\square)}s_\mu(\bm t).
\ee
Now compare this equality with \eqref{eq-lem-L-1+t1},
and we obtain Lemma \ref{lem-eval-schur}.

\subsection{Proof of Lemma \ref{lem-ak-explicit}}
\label{app-pflem5.2}

In this proof, we need to remember the integer $r$ we choose,
and denote $a_k(\bm v)$ for $Z_{(r)}(s; v_1,\cdots,v_r ;\bm p)$ by $a_{(r),k}(\bm v)$.
Recall that $a_{(r),k}(\bm v)$ is determined by:
\be
\label{eq-pf-akr-def}
\sum_{k=1}^{r+1} ka_{(r),k}(\bm v) \cdot \prod_{j=0}^{k-2}(x-j)
= \prod_{i=1}^r (x+ v_i)
= \sum_{j=0}^r x^{r-j} \cdot e_j(\bm v).
\ee
Take $x=0$ in this equality, we obtain:
\be
\label{eq-a1=er}
a_{(r),1}(\bm v) = v_1 v_2\cdots v_r =e_r(\bm v),
\ee
for every $r$.
Then we take $x=1$ and obtain:
\be
a_{(r),2}(\bm v) = \half\Big(
(1+v_1)(1+v_2)\cdots (1+v_r) - a_{(r),1}(\bm v)
\Big)
= \half \sum_{i=0}^{r-1} e_i(\bm v).
\ee
And Similarly, taking $x=2$ gives:
\be
\label{eq-a3=e}
\begin{split}
a_{(r),3}(\bm v) =& \frac{1}{3}\sum_{i=0}^{r-1}(2^{r-i-1}-1)e_i(\bm v).
\end{split}
\ee
Repeat this procedure by taking $x=3,4,5,\cdots$ successively,
and inductively we see that:
\begin{itemize}
\item[1)]
Every $a_{(r),k}(\bm v)$ is a linear combination of $\{e_i(\bm v)\}_{i=0}^r$.
\item[2)]
Fix two integers $j\geq 0$ and $k\geq 1$.
When $r>\max\{j,k\}$,
the coefficient of $e_{r-j}(\bm v)$ in the expansion of $a_{(r),k}$ is independent of $r$.
More precisely,
if $a_{(r),k}(\bm v)$ is of the form:
\begin{equation*}
a_{(r),k}(\bm v) = C_r e_r(\bm v) + C_{r-1} e_{r-1}(\bm v) + \cdots + C_{0} e_{0}(\bm v),
\end{equation*}
then $a_{(r+1),k}(\bm v)$ must be of the form:
\begin{equation*}
a_{(r),k}(\bm v) = C_r e_{r+1}(\bm v) + C_{r-1} e_{r}(\bm v) + \cdots + C_{0} e_{1}(\bm v) +C e_0(\bm v)
\end{equation*}
for some constant $C$.
\end{itemize}
Thus to prove the lemma,
it suffices to show that the coefficient of $e_0(\bm v) = 1$ in $a_{(r),k}(\bm v)$ is:
\begin{equation*}
\frac{1}{k} \sum_{i=1}^{k-1} \frac{(-1)^{i+k-1} \cdot i^r}{i!\cdot (k-1-i)!}.
\end{equation*}
We prove this by induction on $k$.
The cases $k=1,2,3$ can be checked directly using \eqref{eq-a1=er}-\eqref{eq-a3=e}.
Now assume this holds for $k=1,2,\cdots,n$ (where $n\leq r$),
and consider the coefficient of $e_0(\bm v)$ in $a_{(r),n+1}(\bm v)$.
We denote this coefficient by $A$.
Now take $x= n$ in \eqref{eq-pf-akr-def} and compute the coefficients of $e_0(\bm v)$,
and we obtain:
\begin{equation*}
\sum_{k=1}^n
\Big( \sum_{i=1}^{k-1} \frac{(-1)^{i+k-1} \cdot i^r}{i!\cdot (k-1-i)!} \Big)
\Big( \prod_{j=0}^{k-2} (n-j) \Big)
+ (n+1)!\cdot A = n^r.
\end{equation*}
Notice that:
\begin{equation*}
\sum_{k=1}^n
\Big( \sum_{i=1}^{k-1} \frac{(-1)^{i+k-1}  i^r}{i! (k-1-i)!} \Big)
 \prod_{j=0}^{k-2} (n-j)
= \sum_{i=1}^{n-1} \frac{(-1)^{i-1}  i^r \cdot n!}{i!}
\sum_{k=i+1}^{n} \frac{(-1)^{k} }{(k-1-i)!  (n-k+1)!},
\end{equation*}
where
\begin{equation*}
\sum_{k=i+1}^{n} \frac{(-1)^{k} }{(k-1-i)! \cdot (n-k+1)!}
= -\frac{(-1)^{n+1}}{(n-i)!\cdot 0!} = \frac{(-1)^n}{(n-i)!},
\end{equation*}
and therefore we have:
\begin{equation*}
A =  \frac{1}{(n+1)!} \Big(
n^r - \sum_{i=1}^{n-1} \frac{(-1)^{i-1} \cdot i^r \cdot n!}{i!}
\cdot \frac{(-1)^n}{(n-i)!} \Big)
= \frac{1}{n+1}
\sum_{i=1}^n \frac{(-1)^{i+n} \cdot i^r}{i!\cdot (n-i)!},
\end{equation*}
which completes the proof.

\end{appendix}

\end{document}